\documentclass[12pt]{article}
\usepackage{color}
\usepackage{graphicx}
\usepackage{subcaption} % 加载 subcaption 宏包，用于子图排版
\usepackage{amssymb,amsmath,graphicx,acronym,amsfonts}
\usepackage[square, comma, sort&compress, numbers]{natbib}
\usepackage{enumerate}
\usepackage{epstopdf}

\newtheorem{corollary}{Corollary}[section]
\newtheorem{theorem}{Theorem}[section]
\newtheorem{lemma}{Lemma}[section]
\newtheorem{definition}{Definition}[section]
\newtheorem{proposition}{Proposition}[section]
\newtheorem{example}{Example}[section]
\newtheorem{assum}{Assumption}[section]
\newtheorem{algo}{Algorithm}[section]
\newtheorem{Remark}{Remark}[section]

\numberwithin{equation}{section}
\def\x{{\bf x}}
\def\y{{\bf y}}

\def\w{{\bf w}}
\def\s{{\bf s}}
\def\u{{\bf u}}
\def\v{{\bf v}}
\def\A{{\cal A}}

\def\bc{\begin{corl}}
\def\bc{\end{corl}}
\def\ba{\begin{algo}}
\def\ea{\end{algo}}
\def\br{\begin{Remark}}
\def\er{\end{Remark}}
\def\bs{\begin{assum}}
\def\es{\end{assum}}
\def\bt{\begin{theorem}}
\def\et{\end{theorem}\vskip 3pt}
\def\bl{\begin{lemma}}
\def\el{\end{lemma}}
\def\ep{\end{proposition}}
\def\bp{\begin{proposition}}
\def\qed{\hfill{$\Box$}\vskip 5pt}
\def\be{\begin{example}}
\def\ee{\end{example}}
\def\bd{\begin{definition}}
\def\ed{\end{definition}}
\def\bc{\begin{corollary}}
\def\ec{\end{corollary}}
\def\proof{\noindent\it Proof. \hspace{1mm}\rm}
\input epsf
\begin{document}
\title{\bf An Efficient Memory Gradient Method for Extreme M-Eigenvalues of Elastic type Tensors\thanks{This work was supported by the National Natural Science Foundation of P.R. China (Grant No.12171064).}}
%\thanks{The first author's work was supported by the National Natural Science Foundation of P.R. China (Grant No.12171064). This work was supported by the Natural Science Foundation of China (Grant No. 12071249), Shandong Provincial Natural Science Foundation of Distinguished Young Scholars(Grant No. ZR2021JQ01), Hong Kong Innovation and Technology Commision (InnoHK Project CIMDA) and Hong Kong Research Grants Council (Project CityU 11204821).}}
%\author{Zhuolin Du\thanks{School of Management Science, Qufu Normal University, Rizhao Shandong, 276800, China. E-mail: duzhuolin728@163.com}
%\quad
%%Chunyan Wang\thanks{School of Management Science, Qufu Normal University, Rizhao Shandong, 276800, China. E-mail: wchunyanwang562@163.com. }
%%\quad
%Haibin Chen
%\thanks{School of Management Science, Qufu Normal University, Rizhao Shandong, 276800, China. E-mail: chenhaibin508@163.com.}
%\quad
%\quad Hong Yan
%\thanks{Centre for Intelligent Multidimensional Data Analysis, Hong Kong Science Park, Shatin, Hong Kong. Department of Electrical Engineering, City University of Hong Kong, Hong Kong, China. E-mail: h.yan@cityu.edu.hk.}

\date{}
\author{ Zhuolin Du, Yisheng Song\thanks{Corresponding author E-mail: yisheng.song@cqnu.edu.cn}\\
	School of Mathematical Sciences,  Chongqing Normal University, \\
	Chongqing, 401331, P.R. China. \\ Email: duzhuolin728@163.com (Du); yisheng.song@cqnu.edu.cn (Song)}
\maketitle
\vspace{-0.6cm}
\begin{abstract}
M-eigenvalues of fourth order hierarchically  symmetric  tensors play a significant role in nonlinear elastic material analysis and quantum entanglement problems. This paper focuses on computing extreme M-eigenvalues for such tensors. To achieve this, we first reformulate the M-eigenvalue problem as a sequence of unconstrained optimization problems by introducing a shift parameter. Subsequently, we develop a memory gradient method specifically designed to approximate these extreme M-eigenvalues. Under this framework, we establish the global convergence of the proposed method. Finally, comprehensive numerical experiments demonstrate the efficacy and stability of our approach.

\smallskip

\noindent{\bf Keywords:} Hierarchically  symmetric tensors; M-eigenvalues; Memory gradient method; Global convergence

\noindent{\bf AMS Subject Classification(2010):} 90C23, 65H17, 15A18, 90C30.

\end{abstract}

\newpage
\section{Introduction}
Tensor eigenvalue problems play a critical role in numerical multilinear algebra \cite{CQ10,DQW13,HQ18,LLK14,QL17,YY11}, with significant applications in areas such as magnetic resonance imaging (MRI) \cite{BV08}, spectral hypergraph theory \cite{LG13}, and automatic control \cite{GC19,GL19,GZ19}. In particular, M-eigenvalue problems for fourth order hierarchically symmetric tensors have attracted considerable attention in recent years due to their wide range of applications, including nonlinear elastic material analysis \cite{DQW13,GM72,HDH09,HXL17,HQ18,PC02,QDH09,WW03} and quantum entanglement problems in physics \cite{DLM07,CDC07,EPR35,S35}. To further investigate properties related to the M-eigenvalues of fourth order hierarchically symmetric tensors, Han-Dai-Qi \cite{HDH09} and Qi-Dai-Han \cite{QDH09} introduced the definitions of M-eigenvalues and M-eigenvectors for this class of tensors.

\bd\label{def1.1} Let $\mathcal{A}=(a_{ijkl})\in\mathbb{R}^{m\times n\times m\times n}$ or $\mathcal{A}=(a_{ikjl})\in\mathbb{R}^{m\times m\times n\times n}$. If for all $i,k\in [m],~j,l\in [n]$, $$a_{ijkl}=a_{kjil}=a_{ilkj}=a_{klij}, $$ where $[n]=\{1, 2,\ldots,n\}$,
then $\mathcal{A}$ is called a {\em\bf  fourth order hierarchically symmetric tensor} (Helbig \cite{FHR1996,H1994}). Such a  tensor is called {\em\bf elastic type tensor} also (Backus \cite{B1970}).

If for all $i,k\in [m],~j,l\in [n]$,  $$a_{ijkl}=a_{kjil}=a_{ilkj},$$
then $\mathcal{A}$ is called a {\em\bf fourth order partially symmetric tensors} (Qi-Dai-Han \cite{QDH09}). \ed

   Obviously, a fourth order hierarchically symmetric tensor is symmetric whenever $a_{ijkl}=a_{ikjl},$ and a fourth order  $n$-dimensional ($m=n$) hierarchically symmetric tensor has $\dfrac{n(n+1)(n^2+n+2)}{8}$  independent components (Backus \cite{B1970}).  When $m=n=3$,  a fourth order hierarchically symmetric tensor has $21$  independent components,  and a fourth order  partially symmetric tensor has $36$  independent components.  For the sake of simplicity, we denote $\mathbb{H}^{m\times n\times m\times n}$ the set of all fourth order hierarchically symmetric tensors.
\bd\label{def1.2} (\cite{HDH09, QDH09}) Let $\mathcal{A}=(a_{ijkl})\in\mathbb{H}^{m\times n\times m\times n}$. If $(\lambda,\x,\y)\in\mathbb{R}\times\mathbb{R}^m\backslash\{{\bf 0}\}\times\mathbb{R}^n\backslash\{{\bf 0}\}$ satisfies the following system
\begin{equation}\label{1.1}
\mathcal{A}\cdot\y\x\y=\lambda\x,~~\mathcal{A}\x\y\x\cdot=\lambda\y,~~\x^\top\x=1,~~\y^\top\y=1,
\end{equation}
where
 $$(\mathcal{A}\cdot\y\x\y)_i=\sum_{\substack{k\in [m]; ~j, l \in [n]}}a_{ijkl}y_j x_ky_l
\quad \mbox{and}\quad(\mathcal{A}\x\y\x\cdot)_l = \sum_{\substack{i,k\in [m];~j \in [n]}} a_{ijkl} x_i y_j x_k,$$
then $\lambda$ is called an {\em\bf M-eigenvalue} of $\mathcal{A}$ with the {\em\bf left M-eigenvector} $\x$ and the {\em\bf right M-eigenvector} $\y$.
\ed

The exploration of M-eigenvalues for fourth order hierarchically symmetric tensors has become an active research area over the past decade \cite{CCW20,CHW22,LLL19,HLX21,HLW20,WSL20,WQZ09}. Significant progress has been made through interval inclusion methods and spectral approximation techniques \cite{DDV00,KR02}. However, substantial computational challenges remain in accurately characterizing these M-eigenvalues. To address this, Wang et al.  \cite{WQZ09} developed a computational approach that reformulates the largest M-eigenvalue problem as a biquadratic homogeneous polynomial optimization over unit spheres; this method also enables verification of strong ellipticity. Building on this foundation, Li et al. \cite{LLL19} derived two specialized M-eigenvalue inclusion intervals for fourth order hierarchically symmetric tensors. Subsequently, Che et al. \cite{CCW20} estimated M-eigenvalues by analyzing distinct M-eigenvector components and established upper bounds for the M-spectral radius of nonnegative hierarchically symmetric tensors. Notably, these bounds were incorporated as parameters in the WQZ algorithm \cite{WQZ09}, demonstrating their utility in iterative methods. A key application of M-eigenvalue sets is verifying positive definiteness \cite{LL15,LLK14,LWZ14}. A significant limitation of existing inclusion sets \cite{CCW20,LLL19}, however, is their confinement to regions near the coordinate axes, rendering them ineffective for determining M-positive definiteness. Addressing this gap, recent work \cite{HLX21}  introduced novel S-type M-eigenvalue inclusion theorems. Complementing these theoretical advances, an alternating shifted power method was developed for computing extremal M-eigenvalues of fourth order hierarchically symmetric tensors \cite{WCW23}.

Beyond iterative methods like the alternating shifted power method, several specialized optimization strategies have recently emerged for tensor eigenvalue problems \cite{HL13,HLQ13,HC15,HCD15,NQ15}. Han \cite{HL13} introduced an unconstrained optimization model for computing generalized eigenpairs of symmetric tensors; applying the BFGS method to this model yields superlinearly convergent sequences, providing a robust solution framework. In practical applications such as MRI \cite{QY10}, the focus often shifts toward computing extremal eigenvalues or local maxima rather than determining the entire spectral set, highlighting the need for efficient, targeted algorithms. However, as tensor order or dimension increases, these optimization problems become large-scale or even very large-scale, posing significant computational challenges. To address this, Li et al. \cite{LG13} proposed a linearly convergent adaptive gradient method. This raises a natural question: Can adaptive gradient methods be extended to fourth order hierarchically symmetric tensors? We note that the steepest descent method \cite{DM02} relies solely on current gradient information to determine the next iterate, potentially underutilizing historical data. In contrast, Narushima et al. \cite{NY06} developed a memory gradient method that incorporates past information to enhance classical gradient descent while circumventing Hessian computations. Motivated by these insights, particularly the untapped potential of historical gradient data for structured tensor problems, we propose a Memory Gradient Method (MGM) for computing the largest M-eigenvalue of fourth order hierarchically symmetric tensors. We establish the global convergence of MGM and present comprehensive numerical results demonstrating its computational efficiency.

The remainder of this paper is organized as follows. Section 2 introduces a transformation that reformulates the M-eigenvalue problem into an unconstrained optimization framework. Building on this foundation, Section 3 develops the Memory Gradient Method (MGM) for computing extremal M-eigenvalues and establishes its global convergence. Section 4 presents comprehensive numerical experiments demonstrating the computational efficiency and performance advantages of MGM compared to existing methods. Finally, Section 5 provides concluding remarks.

\setcounter{equation}{0}
\section{Unconstrained framework for M-eigenvalues of elastic type tensors}

Let $\mathbb{R}^n$ be the $n$ dimensional real Euclidean space. Generally, scalars are denoted by lowercase letters such as $a, b$ and vectors are denoted by bold lowercase letters such as $\x,\y$. Furthermore, matrices are denoted by capital letters such as ${\bf A}$, ${\bf B}$ and tensors are denoted by calligraphic capital letters such as $\mathcal{A},\mathcal{B}$. The superscript $^\top$ stands for the transpose of vectors and matrices. Let $\mathbb{H}^{m\times n\times m\times n}$ denote the set of all four order hierarchically symmetric tensors. For any $\mathcal{A}=(a_{ijkl})\in\mathbb{H}^{m\times n\times m\times n}$, the corresponding gradient is as follows:
$$\nabla_\x(\mathcal{A}\x\y\x\y)=2\mathcal{A}\cdot\y\x\y,~~~~\nabla_\y(\mathcal{A}\x\y\x\y)=2\mathcal{A}\x\y\x\cdot.$$
For the sake of simple, $\mathcal{A}\cdot\y\x\y$ and $\mathcal{A}\x\y\x\cdot$ are also denoted by $\mathcal{A}\y\x\y$ and $\mathcal{A}\x\y\x$ respectively.

Motivated by the work of Auchmuty \cite{AG91}, we consider the following unconstrained optimization problem, which has close relationship with the M-eigenvalue problem (see Theorem \ref{th2.1}):
\begin{equation}\label{2.1}
\min\limits_{\substack{\x\in\mathbb{R}^m \\ \y\in\mathbb{R}^n}}f(\x,\y)=\frac{1}{4}(\x^\top\x)^2(\y^\top\y)^2-\frac{1}{2}\mathcal{A}\x\y\x\y.
\end{equation}
Then, the gradient of $f(\x,\y)$ with respect to $\x$ and $\y$ are given by
\begin{equation}\label{2.2}
g_1(\x,\y)=\nabla_\x f(\x,\y)=(\x^\top\x)\x(\y^\top\y)^2-\mathcal{A}\y\x\y,
\end{equation}
\begin{equation}\label{2.3}
g_2(\x,\y)=\nabla_\y f(\x,\y)=(\x^\top\x)^2(\y^\top\y)\y-\mathcal{A}\x\y\x.
\end{equation}
Moreover, we have the following results.

\bt\label{th2.1}
Let $\mathcal{A}\in \mathbb{H}^{m\times n\times m\times n}$ be a nonzero tensor. Assume that $\lambda^\ast$ is the largest M-eigenvalue of $\mathcal{A}$. Then we have the following results

(1) The problem \eqref{2.1} has a global minimum.

(2) Let $\x\in\mathbb{R}^m\setminus\{\bf0\}$ and $\y\in\mathbb{R}^n\setminus\{\bf0\}$ be critical points of \eqref{2.1}. Then $\lambda=(\x^\top\x)(\y^\top\y)$ is a positive M-eigenvalue of $\A$, with the associated M-eigenvectors $\u=\x/\|\x\|$ and $\v=\y/\|\y\|$.

(3) If $\lambda^*>0$, then $f(\x,\y)$ attains its global minimum $f_{\min}=-\frac{1}{4}(\lambda^\ast)^2$ at nonzero critical points $(\x,\y)$.

%(4) If $\lambda^\ast\leq0$, the only critical points of $f(\x, \y)$ are $(\bf 0,\y)$ and $(\x, \bf 0)$, for any $\x\in\mathbb{R}^m, \y\in\mathbb{R}^n $. Moreover, $f(\x,\y)$ attains its global minimum at these critical points.
\et

\proof (1) Since the unit sphere is compact and the function $(\u,\v)\mapsto\left|\mathcal{A}\u\v\u\v\right|$ is continuous, denote $C=\max_{\|\u\| = 1, \|\v\| = 1} \left|\mathcal{A}\u\v\u\v\right|$, and the optimal value $C$ is attainable, i.e., there exist unit vectors $\u_0,\v_0$ such that $C=\left|\mathcal{A}\u_0\v_0\u_0\v_0\right|$.
Since $\mathcal{A}$ is a nonzero tensor, we have $C>0$. Now, we consider two cases based on the sign of $\mathcal{A}\u\v\u\v$.
Let $t=\|\x\|^2\|\y\|^2$, if $t=0$, then $f(\x,\y)=0$. Otherwise, we obtain that
$$
f(\x,\y)=\frac{1}{4}(\x^\top\x)^2(\y^\top\y)^2-\frac{1}{2}\mathcal{A}\x\y\x\y
\geq\frac{1}{4}\|\x\|^4\|\y\|^4-\frac{C}{2}\|\x\|^2\|\y\|^2
=\frac{1}{4}t^2-\frac{C}{2}t.
$$
By direct computation, we know that $f(\x,\y)\geq -\frac{1}{4}C^2$, and the equality holds when $t=C$.

(2) By (\ref{2.2})-(\ref{2.3}), the critical points of $f(\x,\y)$ satisfy
\begin{equation}\label{a}
\mathcal{A}\y\x\y=(\x^\top\x)(\y^\top\y)^2\x~ \mbox{and}~\mathcal{A}\x\y\x=(\x^\top\x)^2(\y^\top\y)\y.
\end{equation}
Denote $\u=\x/\|\x\|,\v=\y/\|\y\|$. Then, combining this with \eqref{a}, it holds that
$$\mathcal{A}\v\u\v=(\x^\top\x)\frac{(\y^\top\y)^2}{\|\y\|^2}\u~ \mbox{and}~\mathcal{A}\u\v\u=\frac{(\x^\top\x)^2}{\|\x\|^2}(\y^\top\y)\v.$$
By Definition \ref{def1.2}, we have that $\lambda=(\x^\top\x)(\y^\top\y)>0$ satisfies $\mathcal{A}\v\u\v=\lambda\u$, $\mathcal{A}\u\v\u=\lambda\v$. Thus, $\u\in\mathbb{R}^m\setminus\{\bf0\}$ and $\v\in\mathbb{R}^n\setminus\{\bf0\}$ are M-eigenvectors corresponding to the positive M-eigenvalue $\lambda$.

(3) Let $\lambda$ be an arbitrary positive M-eigenvalue of $\mathcal{A}$. By Definition \ref{def1.2}, there exist unit vectors $\u,\v$ such that $\mathcal{A}\v\u\v=\lambda\u~ \mbox{and}~\mathcal{A}\u\v\u=\lambda\v$. Denote $\x=p\u$ and $\y=q\v$ for some scalars $p,q\in\mathbb{R}\setminus\{0\}$.
The gradients of $f(\x, \y)$ are derived as follows
$$
\begin{aligned}
\nabla_{\x}f(\x,\y)&=(\x^\top\x)(\y^\top\y)^2\x-\mathcal{A}\y\x\y\\&=p^2q^4p\u-\mathcal{A}(q\v)(p\u)(q\v)\\
&=p^3q^4\u-pq^2\mathcal{A}\v\u\v=p^3q^4\u-pq^2\lambda\u\\
&=pq^2(p^2q^2-\lambda)\u
\end{aligned}
$$
$$
\begin{aligned}
\nabla_{\y}f(\x,\y)&=(\x^\top\x)^2(\y^\top\y)\y-\mathcal{A}\x\y\x\\&=p^4q^2q\v-\mathcal{A}(p\u)(q\v)(p\u)\\
&=p^4q^3\v-p^2q\mathcal{A}\u\v\u=p^4q^3\v-p^2q\lambda\v\\
&=p^2q(p^2q^2-\lambda)\v.
\end{aligned}
$$
Note that $\x^\top\x=(p\u)^\top(p\u)=p^2$, and $\y^\top\y=(q\v)^\top(q\v)=q^2$. Choosing $p,q$ such that $p^2q^2=\lambda$, we have $\nabla_{\x}f(\x,\y)=0$ and $\nabla_{\y}f(\x,\y)=0$. Consequently, $(\x,\y)$ is a nonzero critical point of $f(\x,\y)$, and satisfying $(\x^\top\x)(\y^\top\y)=p^2q^2=\lambda$.
At this critical point, we also have $\mathcal{A}\x\y\x\y=(\x^\top\x)^2(\y^\top\y)^2=\lambda^2$.
Therefore, $$f(\x,\y)=-\frac{1}{4}(\x^\top\x)^2(\y^\top\y)^2=-\frac{1}{4}\lambda^2\geq -\frac{1}{4}(\lambda^\ast)^2.$$
This implies that the global minimum is
$$
f_{\min} = -\frac{1}{4}(\lambda^\ast)^2,
$$
attained at a critical point corresponding to $\lambda^\ast = (\x^\top\x)(\y^\top\y)$.
%(4) Since the largest M-eigenvalue satisfies $\lambda^\ast\leq0$, $\lambda=(\x^\top\x)(\y^\top\y)$ does not hold for any M-eigenvectors associated with an M-eigenvalue $\lambda$, as $(\x^\top\x)(\y^\top\y)=\|\x\|^2\|\y\|^2>0$ for any non-zero $\x\in\mathbb{R}^m$ and $\y\in\mathbb{R}^n$. Therefore, by Theorem \ref{th2.1}, $(\bf 0,\y)$ and $(\x, \bf 0)$ are only critical points and the unique global minimum.
\qed

\begin{Remark}
If all M-eigenvalues of $\mathcal{A}$ are non-positive, then $f(\x, \y)$ has no nonzero critical points. Indeed, the existence of a nonzero critical point $(\x, \y)$ would imply that $\lambda = (\x^\top\x)(\y^\top\y) > 0$ is an M-eigenvalue. Thus, in this case, the set of critical points is restricted to trivial solutions of the form $(\mathbf{0}, \y)$ or $(\x, \mathbf{0})$, which do not correspond to any M-eigenvalue.
\end{Remark}

To address the case where $\mathcal{A}$ has only non-positive M-eigenvalues, we introduce the following shifted problem with a parameter $t>0$ as follows
\begin{equation}\label{2.5}
\mathop{\min}\limits_{\substack{\x\in\mathbb{R}^m\\ \y\in\mathbb{R}^n}}f_t(\x,\y)=\frac{1}{4}(\x^\top\x)^2(\y^\top\y)^2-\frac{1}{2}\mathcal{A}\x\y\x\y-\frac{t}{2}(\x^\top\x)(\y^\top\y).
\end{equation}
For a sufficiently large $t$, it holds that $f_t(\x, \y) < 0$ for any nonzero vectors $(\x, \y)$.
On the other hand, $f_t({\bf0},\y)=0$ and $f_t(\x,{\bf0})=0$ for all $\x\in\mathbb{R}^m$ and $\y\in\mathbb{R}^n$. Therefore, points of the form $(\bf0,\y)$ or $(\x,\bf0)$ are maximizers of $f_t(\x,\y)$, while the global minimum is strictly negative and attained at a nonzero point. The gradients of $f_t(\x,\y)$ are given by
\begin{equation}\label{2.6}
g_{t_1}(\x,\y)=\nabla_\x f_t(\x,\y)=(\x^\top\x)(\y^\top\y)^2\x-\mathcal{A}\y\x\y-t(\y^\top\y)\x,
\end{equation}
\begin{equation}\label{2.7}
g_{t_2}(\x,\y)=\nabla_\y f_t(\x,\y)=(\x^\top\x)^2(\y^\top\y)\y-\mathcal{A}\x\y\x-t(\x^\top\x)\y.
\end{equation}

It is clear that for all $\x\in\mathbb{R}^m$ and $\y\in\mathbb{R}^n$, the points $(\bf0,\y)$ and $(\x,\bf0)$ are critical points of $f_t(\x,\y)$. More importantly, any nonzero critical point $(\x,\y)$ of \eqref{2.5} corresponds to an M-eigenvalue $\lambda=({\x}^\top\x)({\y}^\top\y)-t$.  Since $f_t(\mathbf{x},\mathbf{y})$ attains negative values for nonzero $(\mathbf{x},\mathbf{y})$, a suitable descent algorithm for \eqref{2.5} is expected to converge to a nonzero critical point.

Therefore, by solving either the original problem \eqref{2.1} or the shifted problem \eqref{2.5}, the M-eigenvalues and their associated M-eigenvectors can be effectively computed. The following algorithm outlines this adaptive procedure.

\medskip
\medskip
\medskip
\vspace{4mm}
\begin{tabular}{@{}l@{}}
\hline
 \multicolumn{1}{c}{\bf Algorithm 1 }\\
\hline
\quad Step 0. Input $\mathcal{A}\in\mathbb{P}^{m\times n\times m \times n}, ~t\geq1,~\rho>1$ and $0<\epsilon\ll 1$, $k=0$. \\
\quad Step 1. Solving problem (\ref{2.1}) to obtain $(\x_k,\y_k)$. Compute $\lambda_k=({\x_k}^\top\x_k)({\y_k}^\top\y_k)$.\\
\quad Step 2. If $\min\{\|\x_k\|,\|\y_k\|\}>\epsilon$, output $(\x_k,\y_k)$ and $\lambda_k$, Stop. Otherwise go to\\ Step 3.\\
\quad Step 3. Solve problem (\ref{2.5}) to obtain $(\x_k,\y_k)$. Compute $\lambda_k=({\x_k}^\top\x_k)({\y_k}^\top\y_k)-t$.\\
\quad Step 4. If $\min\{||\mathbf{x}_k||, ||\mathbf{y}_k||\} > \epsilon$, output $(\mathbf{x}_k, \mathbf{y}_k)$ and $\lambda_k$, Stop. Otherwise, update \\$t = \rho t$, and repeat Step 3.\\
\hline
\end{tabular}
\vspace{4mm}
\medskip

\begin{Remark}
The iterative process between Step 2 and Step 3 constitutes an inner loop. Since any suitable descent algorithm applied to (\ref{2.5}) is guaranteed to converge to a nonzero critical point provided that $t$ is sufficiently large, this inner loop ensures finite termination. Consequently, Algorithm 1 is well-defined.
\end{Remark}

\begin{Remark}
In practice, the same unconstrained optimization method should be applied to solve both the original problem (\ref{2.1}) and the shifted problem (\ref{2.5}). In the subsequent section, we introduce a specialized memory gradient method tailored for this class of problems.
\end{Remark}

\section{A memory gradient method and its convergence}
In this section, we present the MGM designed to solve the unconstrained optimization problems formulated in Section 2. Specifically, MGM is employed as the primary solver in Steps 1 and 3 within the adaptive framework of Algorithm 1.

For simplicity, let $\mathbf{z} = (\mathbf{x}^\top, \mathbf{y}^\top)^\top \in \mathbb{R}^{m+n}$ and $\Phi(\mathbf{z}): \mathbb{R}^{m+n} \to \mathbb{R}$ denote the objective function (representing either $f(\mathbf{x},\mathbf{y})$ or $f_t(\mathbf{x},\mathbf{y})$). Correspondingly, let $g(\mathbf{z}) = \nabla \Phi(\mathbf{z})$ denote the gradient. The iterative update scheme of the proposed method is given by
$$
{\bf z}_{k+1}={\bf z}_k+\alpha_k{\bf d}_{k},
$$
where $\alpha_k>0$ is a step size determined by a line search procedure. The search direction ${\bf d}_k$ is computed as follows for $k\geq1$
\begin{equation}\label{3.1}
{\bf d}_k=-\gamma_k g({\bf z}_k)+\frac{1}{N}\sum_{i=1}^N\beta_{k_i}{\bf d}_{k-i},
\end{equation}
where $ \beta_{k_i} \in \mathbb{R}$ for $i=1,\ldots,N$ and $ \gamma_k> 0$. For the case $k<N$, \eqref{3.1} is interpreted as ${\bf d}_k=-\gamma_k g({\bf z}_k)+\frac{1}{N}\sum_{i=1}^N\beta_{k}{\bf d}_{k-i}$. The initial search direction is set as the steepest descent direction with a scaling parameter $\gamma_0>0$, specifically, ${\bf d}_{0}=-\gamma_0 g({\bf z}_{0})$. Note that the parameters are different from those proposed by Miele et al. \cite{MC69}. To ensure the search direction satisfies descent property, we define the parameters as follows
\begin{equation}\label{3.3}
\beta_{k_i} = \|g({\bf z}_k)\|^2 \varphi_{k_i}^{\dagger}, \quad \text{where }
a^{\dagger} =
\begin{cases}
0 & \text{if } a = 0, \\
\dfrac{1}{a} & \text{otherwise}.
\end{cases}
\end{equation}
Here, the scalars $\varphi_{k_i}$ are chosen to satisfy
\begin{equation}\label{3.2}
\begin{cases}
\varphi_{k_1} > \max\left\{ \dfrac{{g({\bf z}_k)}^\top {\bf d}_{k-1}}{\gamma_k},\ 0 \right\}, & i = 1, \\
\varphi_{k_i} \geq \max\left\{ \dfrac{{g({\bf z}_k)}^\top {\bf d}_{k-i}}{\gamma_k},\ 0 \right\}, & i = 2, \ldots, N.
\end{cases}
\end{equation}
It is important to note that $a^{\dagger}a \leq 1 $, and from \eqref{3.3} and \eqref{3.2}, it follows that $\beta_{k_1}>0 $ and $ \beta_{k_i}\geq 0 $ for $i= 2,\ldots, N $.

Building on these properties, the following theorem establishes the descent property of the proposed search direction.

\bt\label{them3.1}
Let the search direction ${\bf d}_k $ be defined by \eqref{3.1}. Suppose that the parameters $\beta_{k_i}$ and $\varphi_{k_i}$ are chosen to satisfy conditions \eqref{3.3} and \eqref{3.2} for all $k$. Then, the descent condition holds for all $k$, i.e., $g(\mathbf{z}_k)^\top \mathbf{d}_k < 0$.
\et
\proof For $k=0$, we have $g({\bf z}_0)^\top {\bf d}_0=-\gamma_0\|g({\bf z}_0)\|^2<0$, which establishes the descent property at the initial step.

For $k\geq1$, we consider the expression $$-\gamma_k\|g({\bf z}_k)\|^2+\beta_{k_i}g({\bf z}_k)^\top {\bf d}_{k-i}, ~i=1,\ldots,N.$$

Case 1: $i=1$.
From \eqref{3.1} and \eqref{3.3}, and considering that $\beta_{k_1},\gamma_k>0$, it holds that
\begin{equation*}
\begin{aligned}
&-\gamma_k\|g({\bf z}_k)\|^2+\beta_{k_1}g({\bf z}_k)^\top {\bf d}_{k-1}\\
&\leq-\gamma_k\|g({\bf z}_k)\|^2+\beta_{k_1}\max \left\{g({\bf z}_k)^\top {\bf d}_{k-1},0\right\} \\
&=-\gamma_k\|g({\bf z}_k)\|^2+\|g({\bf z}_k)\|^2\varphi_{k_1}^\dagger\max\left\{g({\bf z}_k)^\top {\bf d}_{k-1},0\right\} \\
&<-\gamma_k\|g({\bf z}_k)\|^2+\gamma_k\|g({\bf z}_k)\|^2\varphi_{k_1}^\dagger\varphi_{k_1} \\
&\leq 0.
\end{aligned}
\end{equation*}
The strict inequality follows from the condition in \eqref{3.2} for $i=1$, and the last inequality uses the property $a^\dagger a \le 1$.

Case 2: $i = 2,\ldots,N$. Similarly,
\begin{equation*}
\begin{aligned}
&-\gamma_k\|g({\bf z}_k)\|^2+\beta_{k_i}g({\bf z}_k)^\top {\bf d}_{k-i}\\
&\leq-\gamma_k\|g({\bf z}_k)\|^2+\beta_{k_i}\max \left\{g({\bf z}_k)^\top {\bf d}_{k-i},0\right\} \\
&=-\gamma_k\|g({\bf z}_k)\|^2+\|g({\bf z}_k)\|^2\varphi_{k_i}^\dagger\max\left\{g({\bf z}_k)^\top {\bf d}_{k-i},0\right\} \\
&\leq-\gamma_k\|g({\bf z}_k)\|^2+\gamma_k\|g({\bf z}_k)\|^2\varphi_{k_i}^\dagger\varphi_{k_i} \\
&\leq 0.
\end{aligned}
\end{equation*}
Combining this with \eqref{3.1}, it follows that
\begin{equation}\label{3.4}
\begin{aligned}
g({\bf z}_k)^\top {\bf d}_k&=-\gamma_k\|g({\bf z}_k)\|^2+\frac{1}{N}\sum_{i=1}^N\beta_{k_i} g({\bf z}_k)^\top {\bf d}_{k-i}\\
&= \frac{1}{N} \sum_{i=1}^N \left(-\gamma_k\|g({\bf z}_k)\|^2+\beta_{k_i}g({\bf z}_k)^\top {\bf d}_{k-i}\right)<0.
\end{aligned}
\end{equation}
Therefore, the descent condition is satisfied and the desired results hold.
\qed

It is worth noting that the index satisfying the strict inequality \eqref{3.2} is not necessarily restricted to $i = 1$. Specifically, the conclusion of the theorem remains valid provided that there exists at least one index $i$ such that
$$
\varphi_{k_i} > \max \left\{ \frac{g({\bf z}_{k}){\bf d}_{k-i}}{\gamma_k}, \, 0 \right\}.
$$
However, utilizing the most recent iterative information is a natural choice.

Next, we focus on establishing the sufficient descent condition, i.e.,
\begin{equation}\label{3.5}
g({\bf z}_k)^\top{\bf d}_{k} \leq-c\|g({\bf z}_k)\|^2, \quad \forall~ k \geq 1
\end{equation}
for some positive constant $c$. The following theorem shows that condition \eqref{3.5} is satisfied by the proposed method under stronger assumptions on $\gamma_k, \varphi_{k_i}$ for $i = 1,\ldots,N $.

\bt\label{them3.2}
Let the search direction ${\bf d}_{k}$ be defined by \eqref{3.1}. Suppose there exists a constant $\bar{\gamma}> 0 $ such that $ \gamma_k\geq \bar{\gamma}$ for all $k$, and $\varphi_{k_i}$ ($i= 1,\ldots, N $) are chosen to satisfy
\begin{equation}\label{3.6}
\begin{cases}
g({\bf z}_k)^\top{\bf d}_{k-1}+\|g({\bf z}_k)\|\|{\bf d}_{k-1}\|<\gamma_k \varphi_{k_1}, & (i = 1), \\[6pt]
g({\bf z}_k)^\top{\bf d}_{k-i}+\|g({\bf z}_k)\|\|{\bf d}_{k-i}\|\leq\gamma_k \varphi_{k_i}, & (i = 2, \ldots, N).
\end{cases}
\end{equation}
Then, the proposed method satisfies the sufficient descent condition \eqref{3.5}.
\et
\proof
Define the index set $\mathcal{I} = \{ i \in \{1, \dots, N\} \mid g(\mathbf{z}_k)^\top \mathbf{d}_{k-i} > 0 \}$. Let $t = |\mathcal{I}|$. It is clear that $0 \le t \le N$.

Note that for $i\notin\mathcal{I}$, we have $g({\bf z}_k)^\top{\bf d}_k\leq0$. By \eqref{3.1},\eqref{3.3}, and \eqref{3.6}, it holds that
$$
\begin{aligned}
g({\bf z}_k)^\top{\bf d}_k &= -\gamma_k \|g({\bf z}_k)\|^2 +\frac{1}{N}\sum_{i=1}^N g({\bf z}_k)^\top{\bf d}_{k-i}\varphi_{k_i}^\dagger \|g({\bf z}_k)\|^2 \\
&\leq -\gamma_k \|g({\bf z}_k)\|^2 + \frac{1}{N} \sum_{i\in \mathcal{I}}\frac{g({\bf z}_k)^\top{\bf d}_{k-i}}{\varphi_{k_i}} \|g({\bf z}_k)\|^2 \\
&\leq -\gamma_k \|g({\bf z}_k)\|^2 + \frac{1}{N}\sum_{i\in \mathcal{I}} \frac{\gamma_k g({\bf z}_k)^\top{\bf d}_{k-i}}{g({\bf z}_k)^\top{\bf d}_{k-i}+\|g({\bf z}_k)\| \|{\bf d}_{k-i}\|} \|g({\bf z}_k)\|^2\\
&\leq -\gamma_k \|g({\bf z}_k)\|^2 + \frac{1}{N}\sum_{i\in\mathcal{I}} \frac{\gamma_k g({\bf z}_k)^\top{\bf d}_{k-i}}{2g({\bf z}_k)^\top{\bf d}_{k-i}} \|g({\bf z}_k)\|^2 \\
&= -\gamma_k\|g({\bf z}_k)\|^2+\frac{1}{N}\sum_{i\in \mathcal{I}} \frac{\gamma_k}{2}\|g({\bf z}_k)\|^2 \\
&= -\gamma_k\|g({\bf z}_k)\|^2 \left(1-\frac{t}{2N} \right). \\
%&\leq -\frac{\gamma_k}{2} \|g({\bf z}_k)\|^2\\
%&\leq -\frac{\bar{\gamma}}{2} \|g({\bf z}_k)\|^2
%=-c\|g({\bf z}_k)\|^2,
\end{aligned}
$$
Since $0 \le t \le N$, we have $1 - \frac{t}{2N} \ge \frac{1}{2}$. Combining this with $\gamma_k \ge \bar{\gamma}$, it follows that
$$
g(\mathbf{z}_k)^\top \mathbf{d}_k \le -\frac{\gamma_k}{2} \|g(\mathbf{z}_k)\|^2 \le -\frac{\bar{\gamma}}{2} \|g(\mathbf{z}_k)\|^2,
$$
which implies that the sufficient descent condition \eqref{3.5} holds with $c = \frac{\bar{\gamma}}{2}$.
\qed

Note that condition \eqref{3.5} is stronger than \eqref{3.4}.
Now, we present the MGM, in which the step size $\alpha_k>0$ is determined by the Wolfe line search conditions
\begin{equation}\label{3.7}
\Phi({\bf z}_k+\alpha_k{\bf d}_k)-\Phi({\bf z}_k)\leq \rho\alpha_k g({\bf z}_k)^\top{\bf d}_k,
\end{equation}
\begin{equation}\label{3.8}
g({\bf z}_k+\alpha_k{\bf d}_k)^\top{\bf d}_k\geq \sigma g({\bf z}_k)^\top{\bf d}_k,
\end{equation}
where $0<\rho<\sigma<1.$

\medskip
\medskip
\medskip
\vspace{4mm}
\begin{tabular}{@{}l@{}}
\hline
 \multicolumn{1}{c}{\bf Algorithm 2. MGM } \\
\hline
{\bf Input:} Function $\Phi(\mathbf{z})$ (problem (2.1) or (2.5)), shift parameter $t \geq 0$ (set $t=0$ for\\ problem (2.1)), $0<\epsilon\ll1$, $\gamma_0>0$, $N>0$, $0<\rho<\sigma<1$.\\
\quad Step 0. Choose a random initial point ${\bf z}_0= ({\bf x}_0^\top, {\bf y}_0^\top)^\top\in\mathbb{R}^{m+n}$. Compute $g(\mathbf{z}_{0})$. If \\$\|g({\bf z}_0)\|\leq\epsilon$, stop; otherwise, set ${\bf d}_0=-\gamma_0 g({\bf z}_0)$ and $k=0$. \\
\quad Step 1. Determine the step size $\alpha_k>0$ satisfying \eqref{3.7} and \eqref{3.8}.\\
\quad Step 2. Update ${\bf z}_{k+1}={\bf z}_k+\alpha_k{\bf d}_k$. If $\|g({\bf z}_{k+1})\|\leq\epsilon$, set ${\bf z}_k={\bf z}_{k+1}$, stop. Otherwise,\\ calculate $\xi_k = \sqrt{\|\mathbf{y}_{k+1}\|/\|\mathbf{x}_{k+1}\|}$.
Update $\mathbf{x}_{k+1} \leftarrow \xi_k \mathbf{x}_{k+1}$ and $\mathbf{y}_{k+1} \leftarrow (1/\xi_k)\mathbf{y}_{k+1}$. Go\\ to Step 3. \\
\quad Step 3. Compute $\gamma_{k+1}$ and $\varphi_{(k+1)_i}$ satisfying (\ref{3.2}), then compute $\beta_{(k+1)_i}$ by (\ref{3.3}).\\
\quad Step 4. Compute ${\bf d}_{k+1}$ by (\ref{3.1}). Set $k\leftarrow k+1$ and return to Step 1.\\
%\quad Step 2. Compute $\gamma_k>0$ and $\varphi_{k_i}$ that satisfy \eqref{3.2}, define $\beta_{k_i}$ by \eqref{3.3} and generate \\${\bf d}_k$ by \eqref{3.1}.\\
%\quad Step 3. Compute $\alpha_k>0$ satisfying \eqref{3.7} and \eqref{3.8}.\\
%\quad Step 4. Update ${\bf z}_{k+1}={\bf z}_k+\alpha_k{\bf d}_k$. Set $k=k+1$. If $\|g({\bf z}_k)\|\leq\varepsilon$, then stop, output \\${\bf z}_k=(\x_k^\top,\y_k^\top)^\top$, otherwise, return Step 2.\\
{\bf Output:} ${\bf z}_k= \left({{\bf x}_k}^\top, {{\bf y}_k}^\top\right)^\top$.
M-eigenvalue $\lambda=({{\bf x}_k}^\top{\bf x}_k)({{\bf y}_K}^\top{\bf y}_k)-t$; associated \\left and right eigenvectors $\x_k/\|\x_k\|$ and $\y_k/\|\y_k\|$.\\
\hline
\end{tabular}
\vspace{4mm}
\medskip
\begin{Remark}
It is crucial to observe that the objective function $\Phi(\mathbf{z})$ is invariant under the transformation $\mathbf{x} \to \xi\mathbf{x}$ and $\mathbf{y} \to \xi^{-1} \mathbf{y}$ for any $\xi \neq 0$. Consequently, the generated sequence $\{\mathbf{z}_k\}$ could theoretically be unbounded where $\|\mathbf{x}\| \to \infty$ and $\|\mathbf{y}\| \to 0$. In Algorithm 2, the rescaling step restricts the iterates to satisfy $\|\mathbf{x}\| = \|\mathbf{y}\|$. This operation guarantees the compactness required for the convergence analysis (see Lemma \ref{lem1}) without altering the objective function value or the directions of the M-eigenvectors.
\end{Remark}

To prove the global convergence of the proposed method, we first establish several key lemmas.

\begin{lemma}
\label{lem1}
%Let $\{\mathbf{z}_k\}$ be the sequence generated by MGM. Let the level set be $\Omega = \{ \mathbf{z} \in \mathbb{R}^{m+n} \mid \Phi(\mathbf{z}) \le c \}$, where $c = \Phi(\mathbf{z}_0) < 0$. Then, $\Omega$ is bounded. Specifically, there exist constants $\eta_1, \eta_2 > 0$ such that for all $\mathbf{z}=(\mathbf{x}^\top, \mathbf{y}^\top)^\top \in \Omega$, it holds that $\|\mathbf{x}\| \le \eta_1$ and $\|\mathbf{y}\| \le \eta_2$.
Let $\{\mathbf{z}_k\}$ be the sequence generated by MGM. Then, the sequence $\{\mathbf{z}_k\}$ is bounded. Specifically, there exist constants $\eta_1, \eta_2 > 0$ such that $\|\mathbf{x}_k\| \le \eta_1$ and $\|\mathbf{y}_k\| \le \eta_2$ for all $k \ge 0$.
\end{lemma}

\proof
Since $\{\mathbf{z}_k\}$ is generated by Algorithm 2, the descent property ensures $\Phi(\mathbf{z}_k) \le \Phi(\mathbf{z}_0)$ for all $k$. Moreover, the algorithm enforces $\|\mathbf{x}_k\| = \|\mathbf{y}_k\|$. Let $s_k = \|\mathbf{x}_k\|\|\mathbf{y}_k\|$. It follows that $\|\mathbf{x}_k\|^2 = \|\mathbf{y}_k\|^2 = s_k$, and consequently $\|\mathbf{z}_k\|^2 = 2s_k$.

Recall the objective function $\Phi(\mathbf{z})$ defined in \eqref{2.5} (which covers the original problem when $t=0$).
$$
\Phi(\mathbf{z}_k) = \frac{1}{4}\|\mathbf{x}_k\|^4\|\mathbf{y}_k\|^4 - \frac{1}{2}\mathcal{A}\mathbf{y}_k\mathbf{x}_k\mathbf{y}_k - \frac{t}{2}\|\mathbf{x}_k\|^2\|\mathbf{y}_k\|^2.
$$

By Theorem 2.1(1), $|\mathcal{A}\mathbf{x}_k\mathbf{y}_k\mathbf{x}_k\mathbf{y}_k| \le C \|\mathbf{x}_k\|^2 \|\mathbf{y}_k\|^2 = C s_k^2$ for some constant $C > 0$. Consequently,
$$
\Phi(\mathbf{z}_k) \ge \frac{1}{4}s_k^4 - \frac{C}{2}s_k^2 - \frac{t}{2}s_k^2
= \frac{1}{4}s_k^4 - \frac{C+t}{2}s_k^2.
$$

Suppose for contradiction that the sequence $\{\mathbf{z}_k\}$ is unbounded. Then there exists a subsequence $\{\mathbf{z}_{k_j}\}$ such that $\|\mathbf{z}_{k_j}\| \to \infty$ as $j \to \infty$, which is equivalent to $s_{k_j} \to \infty$.
Thus, we have
$$
\lim_{j \to \infty} \Phi(\mathbf{z}_{k_j}) \ge \lim_{s_{k_j} \to \infty} \left( \frac{1}{4}s_{k_j}^4 - \frac{C+t}{2}s_{k_j}^2 \right) = +\infty.
$$
This implies that for sufficiently large $j$, $\Phi(\mathbf{z}_{k_j}) > \Phi(\mathbf{z}_0)$, which leads to a contradiction. Therefore, the sequence $\{\mathbf{z}_k\}$ is bounded. Consequently, there exist constants $\eta_1, \eta_2 > 0$ such that $\|\mathbf{x}_k\| \le \eta_1$ and $\|\mathbf{y}_k\| \le \eta_2$ for all $k \ge 0$.
\qed

For the subsequent convergence analysis, define
\begin{equation}\label{1}
\mathcal{D} = \left\{ \mathbf{z}=(\mathbf{x}^\top, \mathbf{y}^\top)^\top \in \mathbb{R}^{m+n} \mid \|\mathbf{x}\| \le \eta_1, \|\mathbf{y}\| \le \eta_2 \right\},
\end{equation}
where $\eta_1$ and $\eta_2$ are constants from Lemma \ref{lem1}.

%However, we restrict our analysis to the level set where $\Phi(\mathbf{z}) \le c$. This contradicts the limit $+\infty$. Thus, $\|\mathbf{z}\|$ cannot approach infinity, implying there exists an upper bound $B_{\max}$ such that $s \le B_{\max}$.
%
%\textbf{Case 2: $\|\mathbf{z}\| \to 0$ (i.e., $s \to 0$).}
%As $s \to 0$, the function value approaches zero:
%$$
%\lim_{\|\mathbf{z}\| \to 0} \Phi(\mathbf{z}) = \lim_{s \to 0} \left( \frac{1}{4}s^4 - \frac{C+t}{2}s^2 \right) = 0.
%$$
%However, the level set is defined by $\Phi(\mathbf{z}) \le c < 0$. Since $0 > c$, the points approaching the origin (where $\Phi \approx 0$) are strictly excluded from $\Omega$. This implies there exists a lower bound $B_{\min} > 0$ such that $s \ge B_{\min}$.
%
%\textbf{Conclusion:}
%Combining Case 1 and Case 2, we have $B_{\min} \le s \le B_{\max}$. Since $\|\mathbf{z}\| = \sqrt{2s}$, it follows that
%$$
%\sqrt{2B_{\min}} \le \|\mathbf{z}\| \le \sqrt{2B_{\max}}.
%$$
%Therefore, the level set $\Omega$ is bounded (and bounded away from zero).

\bl\label{lema3.1} If $\mathcal{A} \in \mathbb{H}^{m \times n \times m \times n}$, then $\mathcal{A}\mathbf{y}\mathbf{x}\mathbf{y}$ and $\mathcal{A}\mathbf{x}\mathbf{y}\mathbf{x}$ are Lipschitz continuous on $\mathcal{D}$.
\el
\proof Let $M = \max_{i,j,k,l} |a_{ijkl}|$. First, we establish the Lipschitz continuity of $\mathcal{A}\y\x\y$. For any $(\x,\y),(\x',\y)\in\mathcal{D}$, we have
\begin{equation}\label{3.9}
\begin{aligned}
\|\mathcal{A}\mathbf{y}\mathbf{x}\mathbf{y} - \mathcal{A}\mathbf{y}\mathbf{x}'\mathbf{y}\|
&= \left\| \left( \sum_{k=1}^m \sum_{j,l=1}^n a_{ijkl} \, y_j \, (x_k - x_k') \, y_l \right)_{i=1}^m \right\|\\
&\le \sqrt{m} M \sum_{j=1}^n |y_j| \sum_{l=1}^n |y_l| \sum_{k=1}^m |x_k - x_k'| \\
&\le \sqrt{m} M \left(\sqrt{n}\|\mathbf{y}\|\right) \left(\sqrt{n}\|\mathbf{y}\|\right) (\sqrt{m}\|\mathbf{x} - \mathbf{x}'\|) \\
&\le mn M \eta_2^2 \|\mathbf{x} - \mathbf{x}'\| = P_1 \|\mathbf{x} - \mathbf{x}'\|,
\end{aligned}
\end{equation}
where $P_1 = mn M \eta_2^2$ is a positive constant.

Next, for $(\mathbf{x}', \mathbf{y}), (\mathbf{x}', \mathbf{y}') \in \mathcal{D}$, utilizing the inequality $|ab - a'b'| \le |a||b-b'| + |b'||a-a'|$, we obtain
\begin{equation}\label{3.11}
\begin{aligned}
\|\mathcal{A}\mathbf{y}\mathbf{x}'\mathbf{y} - \mathcal{A}\mathbf{y}'\mathbf{x}'\mathbf{y}'\|
&= \left\| \left( \sum_{k=1}^m \sum_{j,l=1}^n a_{ijkl} \, x'_k \, (y_jy_l-y'_jy'_l) \, \right)_{i=1}^m \right\|\\
&\le \sqrt{m} M \sum_{k=1}^m |x_k'| \sum_{j,l=1}^n \left(|y_j||y_l - y_l'| + |y_l'||y_j - y_j'|\right) \\
&\le \sqrt{m} M \left(\sqrt{m}\eta_1\right) \left(2\sqrt{n}\eta_2 \sqrt{n}\|\mathbf{y} - \mathbf{y}'\|\right) \\
&= 2mn M \eta_1 \eta_2 \|\mathbf{y} - \mathbf{y}'\| = P_2 \|\mathbf{y} - \mathbf{y}'\|,
\end{aligned}
\end{equation}
where $P_2 = 2mn M \eta_1 \eta_2$ is a positive constant. Thus, $\mathcal{A}\y\x\y$ is Lipschitz continuous with respect to both $\mathbf{x}$ and $\mathbf{y}$.

Similarly, we analyze $\mathcal{A}\mathbf{x}\mathbf{y}\mathbf{x}$. For $(\mathbf{x}, \mathbf{y}), (\mathbf{x}', \mathbf{y}) \in \mathcal{D}$, it follows that
\begin{equation}\label{3.10}
\begin{aligned}
\|\mathcal{A}\x\y\x-\mathcal{A}\x'\y\x'\|
&= \left\| \left( \sum_{i,k=1}^m \sum_{j=1}^n a_{ijkl} \, y_j \, (x_ix_k -x'_i x_k') \right)_{l=1}^n \right\|\\
&\le \sqrt{n} M \sum_{j=1}^n |y_j| \sum_{i,k=1}^m \left(|x_k||x_i - x_i'| + |x_i'||x_k - x_k'|\right) \\
&\le \sqrt{n} M \left(\sqrt{n}\eta_2\right) \left(2\sqrt{m}\eta_1 \sqrt{m}\|\mathbf{x} - \mathbf{x}'\|\right) \\
&= P_2 \|\mathbf{x} - \mathbf{x}'\|.
\end{aligned}
\end{equation}

For $(\x',\y),(\x',\y')\in\mathcal{D}$, it holds that
\begin{equation}\label{3.12}
\begin{aligned}
\|\mathcal{A}\x'\y\x'-\mathcal{A}\x'\y'\x'\|
&= \left\| \left( \sum_{i,k=1}^m \sum_{j=1}^n a_{ijkl} \, x'_i \, (y_j -y_j') x'_k \right)_{l=1}^n \right\|\\
&\le \sqrt{n} M \sum_{i,k=1}^m |x_i' x_k'| \sum_{j=1}^n |y_j - y_j'| \\
&\le \sqrt{n} M \left(m\eta_1^2\right) \left(\sqrt{n}\|\mathbf{y} - \mathbf{y}'\|\right)\\
&= mn M \eta_1^2 \|\mathbf{y} - \mathbf{y}'\| = P_3 \|\mathbf{y} - \mathbf{y}'\|,
\end{aligned}
\end{equation}
where $P_3 = mn M \eta_1^2$ is a positive constant.
By \eqref{3.9}-\eqref{3.12}, we conclude that both $\mathcal{A}\y\x\y$ and $\mathcal{A}\x\y\x$ are Lipschitz continuous on $\mathcal{D}$.
\qed

\begin{lemma}
\label{lema3.3}
Let $\mathcal{A} \in \mathbb{H}^{m \times n \times m \times n}$. Then, the gradient $g_t(\mathbf{z})$ of the shifted problem (2.5) is Lipschitz continuous on $\mathcal{D}$. Specifically, there exists a positive constant $L_t$ such that for all $\mathbf{z}, \mathbf{z}' \in \mathcal{D}$, the following inequality holds
\begin{equation}\label{3.11}
\|g_t(\mathbf{z}) - g_t(\mathbf{z}')\| \le L_t \|\mathbf{z} - \mathbf{z}'\|.
\end{equation}
In particular, when $t=0$, the same conclusion holds for \eqref{2.1}.
\end{lemma}

\proof
First, we estimate the Lipschitz constant for $g_{t_1}(\x,\y)$, it follows from \eqref{2.6} and Lemma \ref{lem1} that for all $(\x,\y)$, $(\x',\y')\in \Omega$, we have

$$
\begin{aligned}
&\|g_{t_1}(\x,\y)-g_{t_1}(\x',\y)\|\\
=&\|(\x^\top\x)(\y^\top\y)^2\x-\mathcal{A}\y\x\y-t(\y^\top\y)\x-(\x'^\top\x')(\y^\top\y)^2\x'+\mathcal{A}\y\x'\y+t(\y^\top\y)\x'\|\\
\leq&\|\mathcal{A}\y\x\y-\mathcal{A}\y\x'\y\|+\|(\x^\top\x)(\y^\top\y)^2\x-(\x'^\top\x')(\y^\top\y)^2
\x'\|+\|t(\y^\top\y)\x'-t(\y^\top\y)\x'\|\\
\leq& P_1\|\x-\x'\|+\|\y\|^4\|(\x^{\top}\x)\x-(\x^{\top}\x)\x'+(\x^{\top}\x)\x'-(\x'^\top\x')\x'\|
+|t|\|\y\|^2\|\x-\x'\|\\
\leq& P_1\|\x-\x'\|+ \|\y\|^4\left(\|\x\|^2\|\x-\x'\|+(\|\x\|+\|\x'\|)\|\x'\|(\|\x\|-\|\x'\|) \right)+|t|\|\y\|^2\|\x-\x'\|\\
\leq& \left( P_1 + \eta_2^4 (\eta_1^2 + \eta_1(\eta_1 + \eta_1)) + t \eta_2^2 \right) \|\mathbf{x} - \mathbf{x}'\| \\
%=&P_1\|\x-\x'\|+3\eta_1^2\eta_2^4\|\x-\x'\|+t\eta_2^2\|\x-\x'\|\\
=&L_{11}\|\x-\x'\|
%& P_1\|\x-\x'\|+P_3\|\x-\x'\|+P_5\|\x-\x'\|\\
%=&(P_1+P_3+P_5)\|\x-\x'\|=
\end{aligned}
$$
where $L_{11}=P_1+3\eta_1^2\eta_2^4+t\eta_2^2$ is a positive constant.
$$
\begin{aligned}
&\|g_{t_1}(\x',\y)-g_{t_1}(\x',\y')\|\\
=&\|(\x'^\top\x')(\y^\top\y)^2\x'-\mathcal{A}\y\x'\y-t(\y^\top\y)\x'-(\x'^\top\x')(\y'^\top\y')^2\x'
+\mathcal{A}\y'\x'\y'+t(\y'^\top\y')\x'\|\\
\leq&\|\mathcal{A}\y\x'\y-\mathcal{A}\y'\x'\y'\|+\|(\x'^\top\x')(\y^\top\y)^2\x'-(\x'^\top\x')(\y'^\top\y')^2
\x'\|+\|t(\y^\top\y)\x'-t(\y'^\top\y')\x'\|\\
\leq&P_2\|\y-\y'\|+\|\x'\|^3\|(\y^\top\y)^2-(\y'^\top\y')^2\|+\|t\x'\|\|\y^\top\y-\y'^\top\y'\| \\
\leq&P_2\|\y-\y'\|+\|\x'\|^3(\|\y\|^2+\|\y'\|^2)(\|\y\|+\|\y'\|)\|\y-\y'\|+t\|\x'\|\|\y+\y'\|\|\y-\y'\|\\
\leq& \left( P_2 + 4\eta_1^3 \eta_2^3 + 2t \eta_1 \eta_2 \right) \|\mathbf{y} - \mathbf{y}'\| = L_{12} \|\mathbf{y} - \mathbf{y}'\|.
\end{aligned}
$$

Therefore, it holds that
\begin{equation}\label{3.14}
\begin{aligned}
\|g_{t_1}(\x,\y)-g_{t_1}(\x',\y')\|&=\|g_{t_1}(\x,\y)-g_{t_1}(\x',\y)+g_{t_1}(\x',\y)-g_{t_1}(\x',\y')\|\\
&\leq\|g_{t_1}(\x,\y)-g_{t_1}(\x',\y')\|+\|g_{t_1}(\x',\y)-g_{t_1}(\x',\y')\|\\
&\leq L_{11}\|\x-\x'\|+L_{12}\|\y-\y'\|.
\end{aligned}
\end{equation}

Similarly, we estimate the Lipschitz constant for $g_{t_2}(\x,\y)$. From \eqref{2.7} and Lemma \ref{lem1}, we have
$$
\begin{aligned}
&\|g_{t_2}(\x,\y)-g_{t_2}(\x',\y)\|\\
=&\|(\x^\top\x)^2(\y^\top\y)\y-\mathcal{A}\x\y\x-t(\x^\top\x)\y-(\x'^\top\x')^2(\y^\top\y)\y
+\mathcal{A}\x'\y\x'+t(\x'^\top\x')\y\|\\
\leq&\|\mathcal{A}\x\y\x-\mathcal{A}\x'\y\x'\|+\|(\x^\top\x)^2(\y^\top\y)\y-(\x'^\top\x')^2(\y^\top\y)\y\|
+\|t(\x^\top\x)\y-t(\x'^\top\x')\y\|\\
\leq& P_2\|\x-\x'\|+ \|\y\|^3\left(\|\x\|^2+\|\x'\|^2 \right)(\|\x\|+\|\x'\|)\|\x-\x'\|+t\|\y\|(\|\x\|+\|\x'\|)\|\x-\x'\|\\
=&L_{21}\|\x-\x'\|,
\end{aligned}
$$
where $L_{21}=P_2+4\sqrt{n}\eta_1^3\eta_2^3+2t\eta_1\eta_2$ is a positive constant.
$$
\begin{aligned}
&\|g_{t_2}(\x',\y)-g_{t_2}(\x',\y')\|\\
=&\|(\x'^\top\x')^2(\y^\top\y)\y-\mathcal{A}\x'\y\x'-t(\x'^\top\x')\y-(\x'^\top\x')^2(\y'^\top\y')\y'
+\mathcal{A}\x'\y'\x'+t(\x'^\top\x')\y'\|\\
\leq&\|\mathcal{A}\x'\y\x'-\mathcal{A}\x'\y'\x'\|+\|(\x'^\top\x')^2(\y^\top\y)\y-(\x'^\top\x')^2(\y'^\top\y')\y'\|
+\|t(\x'^\top\x')\y-t(\x'^\top\x')\y'\|\\
%\leq& P_3\|\y-\y'\|+\|\x\|^4\|(\y^{\top}\y)\y-(\y^{\top}\y)\y'+(\y^{\top}\y)\y'-(\y'^\top\y')\y'\|
%+|t|\|\x\|^2\|\y-\y'\|\\
\leq& P_3\|\y-\y'\|+ \|\x'\|^4\left(\|\y\|^2\|\y-\y'\|+(\|\y\|+\|\y'\|)\|\y'\|(\|\y-\y'\|) \right)+t\|\x'\|^2\|\y-\y'\|\\
%\leq&P_3\|\y-\y'\|+\eta_1^4(\eta_2^2\|\y-\y'\|+2\eta_2^2\|\y-\y'\|)+|t|\|\x\|^2\|\y-\y'\|\\
=&L_{22}\|\y-\y'\|,
\end{aligned}
$$
where $L_{22}=P_3+3\eta_1^4\eta_2^2+t\eta_1^2$ is a positive constant.

Therefore, it holds that
\begin{equation}\label{3.15}
\begin{aligned}
\|g_{t_2}(\x,\y)-g_{t_2}(\x',\y')\|&=\|g_{t_2}(\x,\y)-g_{t_2}(\x',\y)+g_{t_2}(\x',\y)-g_{t_2}(\x',\y')\|\\
&\leq\|g_{t_2}(\x,\y)-g_{t_2}(\x',\y')\|+\|g_{t_2}(\x',\y)-g_{t_2}(\x',\y')\|\\
&\leq L_{21}\|\x-\x'\|+L_{22}\|\y-\y'\|.
\end{aligned}
\end{equation}

Combining the above, let $M_1 = \sqrt{2L_{11}^2 + 2L_{21}^2}$ and $M_2 = \sqrt{2L_{12}^2 + 2L_{22}^2}$. Using the inequality $(a+b)^2 \le 2(a^2+b^2)$, we have
\begin{align*}
\|g_t(\mathbf{z}) - g_t(\mathbf{z}')\|^2 &= \|g_{t_1}(\mathbf{z}) - g_{t_1}(\mathbf{z}')\|^2 + \|g_{t_2}(\mathbf{z}) - g_{t_2}(\mathbf{z}')\|^2 \\
&\le 2\left(L_{11}^2 + L_{21}^2\right)\|\mathbf{x}-\mathbf{x}'\|^2 + 2\left(L_{12}^2 + L_{22}^2\right)\|\mathbf{y}-\mathbf{y}'\|^2 \\
&\le \max\left(M_1^2, M_2^2\right) \left(\|\mathbf{x}-\mathbf{x}'\|^2 + \|\mathbf{y}-\mathbf{y}'\|^2\right) \\
&= L_t^2 \|\mathbf{z}-\mathbf{z}'\|^2.
\end{align*}
Taking the square root completes the proof with $L_t = \max\left(M_1, M_2\right)$.\qed
%By \eqref{3.14} and \eqref{3.15}, it holds that
%$$
%\begin{aligned}
%\|g_t(\x,\y)-g_t(\x',\y')\|&=\sqrt{\|g_{t_1}(\x,\y)-g_{t_1}(\x',\y')\|^2+\|g_{t_2}(\x,\y)-g_{t_2}(\x',\y')\|^2}\\
%&\leq\sqrt{(L_{11}\|\x-\x'\|+L_{12}\|\y-\y'\|)^2+(L_{21}\|\x-\x'\|+L_{22}\|\y-\y'\|)^2}\\
%&\leq\sqrt{2L_{11}\|\x-\x'\|^2+2L_{12}\|\y-\y'\|^2+2L_{21}\|\x-\x'\|^2+2L_{22}\|\y-\y'\|^2}\\
%&=\sqrt{\left(2{L_{11}}^2+2{L_{12}}^2\right)\|\x-\x'\|^2+\left(2{L_{21}}^2+2{L_{22}}^2\right)\|\y-\y'\|^2},
%\end{aligned}
%$$
%where the last inequality follows from $(a+b)^2\leq2(a^2+b^2)$.
%Denote $M_1=\sqrt{2{L_{11}}^2+2{L_{12}}^2}$ and $M_2=\sqrt{2{L_{21}}^2+2{L_{22}}^2}$, then
%$$
%\begin{aligned}
%\|g_t(\x,\y)-g_t(\x',\y')\|&\leq\sqrt{M_1^2\|\x-\x'\|^2+M_2^2\|\y-\y'\|^2}\\
%&\leq\sqrt{M^2_{max}\left(\|\x-\x'\|^2+\|\y-\y'\|^2\right)}\\
%&\leq\sqrt{\left(M_1^2+M_2^2\right)\left(\|\x-\x'\|^2+\|\y-\y'\|^2\right)}\\
%&=L\sqrt{\|\x-\x'\|^2+\|\y-\y'\|^2}=L\|(\x,\y)-(\x',\y')\|,
%\end{aligned}
%$$
%where $M_{max}=\max\{M_1,M_2\}$, $L=\sqrt{M_1^2+M_2^2}$.

Based on the lemmas established above, we now state the following well-known lemma which was proved by Zoutendijk \cite{ZO70}, which applies to general iterative methods.

\bl\label{lema3.4} Let the sequences $\{{\bf z}_k\}$ and $\{{\bf d}_k\}$ be generated by MGM. Then,
$$
\sum_{k=0}^\infty \frac{\left(g({\bf z}_k)^\top{\bf d}_k\right)^2}{\|{\bf d}_k\|^2}<\infty.
$$
\el
\proof From the Wolfe line search condition \eqref{3.8} and Lemma \ref{lema3.3}, it follows that
$$
(\sigma-1)g({\bf z}_k)^\top{\bf d}_k\leq\left(g({\bf z}_{k+1})-g({\bf z}_k)\right)^\top{\bf d}_k\leq L_t\alpha_k\|{\bf d}_k\|^2,
$$
which implies that
$$\alpha_k\geq\frac{(\sigma-1)g({\bf z}_k)^\top{\bf d}_k}{L\|{\bf d}_k\|^2}.$$
Given that $g({\bf z}_k)^\top{\bf d}_k<0$, it follows from \eqref{3.7} that
\begin{equation}\label{3.16}
\Phi(\mathbf{z}_{k+1}) - \Phi(\mathbf{z}_k) \le \rho \alpha_k g(\mathbf{z}_k)^\top \mathbf{d}_k \le \frac{\rho(\sigma - 1)}{L_t} \frac{\left(g(\mathbf{z}_k)^\top \mathbf{d}_k\right)^2}{\|\mathbf{d}_k\|^2}.
\end{equation}
It follows from \eqref{3.16} that the sequence $\{\Phi({\bf z}_k)\}$ is non-increasing. Moreover, by Theorem \ref{th2.1}, $\Phi({\bf z}_k)$ is bounded below. Thus, $\{\Phi({\bf z}_k)\}$ converges.
By taking a series for both sides of (\ref{3.16}), it holds that
$$
\sum_{k=0}^\infty \frac{\left(g({\bf z}_k)^\top{\bf d}_k\right)^2}{\|{\bf d}_k\|^2}<\infty.
$$
\qed

By Lemma \ref{lema3.4}, we establish the following theorem.
\bt\label{them3.3} Let the sequence $\{{\bf z}_k\}$ be generated by MGM. Then, the method either terminates at a stationary point or converges in the sense that
\begin{equation}\label{3.17}
\liminf\limits_{k\to\infty}\|g({\bf z}_k)\|=0
\end{equation}
\et
\proof
From \eqref{3.2}, it follows that
$$
\begin{aligned}
\sum_{i=1}^N \left(\gamma_k \varphi_{k_i}-g({\bf z}_k)^\top{\bf d}_{k-i}\right)\beta_{k_i}
&>(\gamma_k \varphi_{k_1}-\gamma_k\varphi_{k_1}) \beta_{k_1}+\sum_{i=2}^N \left(\gamma_k \varphi_{k_i}- g({\bf z}_k)^\top{\bf d}_{k-i}\right) \beta_{k_i} \\
&\geq \sum_{i=2}^N (\gamma_k \varphi_{k_i}-\gamma_k\varphi_{k_i}) \beta_{k_i}
= 0.
\end{aligned}
$$
Since all the assumptions of Theorem \ref{them3.1} are satisfied, search directions are descent, which implies that
$$
g({\bf z}_k)^\top{\bf d}_k< 0 \quad \text{for all } k.
$$
Therefore, it follows from \eqref{3.1} and \eqref{3.3} that
\begin{equation}\label{3.18}
\begin{aligned}
\left|g({\bf z}_k)^\top{\bf d}_k\right| &= -g({\bf z}_k)^\top{\bf d}_k \\
&= \gamma_k \|g({\bf z}_k)\|^2 - \frac{1}{N} \sum_{i=1}^N \beta_{k_i} g({\bf z}_k)^\top{\bf d}_{k-i} \\
&= \frac{1}{N} \sum_{i=1}^N \left( \gamma_k \|g({\bf z}_k)\|^2 - \beta_{k_i}g({\bf z}_k)^\top{\bf d}_{k-i} \right)  \\
&= \frac{1}{N} \sum_{i=1}^N \left( \gamma_k \varphi_{k_i}-g({\bf z}_k)^\top{\bf d}_{k-i}\right) \beta_{k_i}
> 0.
\end{aligned}
\end{equation}
Rewrite \eqref{3.3} as
$$
{\bf d}_k+\gamma_k g({\bf z}_k)=\frac{1}{N} \sum_{i=1}^N \beta_{k_i}{\bf d}_{k-i}.
$$
Squaring both sides and simplifying
$$
\|{\bf d}_k\|^2 = \left\| \frac{1}{N} \sum_{i=1}^N \beta_{k_i}{\bf d}_{k-i} \right\|^2-2\gamma_k g({\bf z}_k)^\top{\bf d}_{k}-\gamma_k^2 \|g({\bf z}_k)\|^2.
$$
Dividing both sides by $\left(g({\bf z}_k)^\top{\bf d}_k\right)^2$ and applying \eqref{3.18}, it holds that
\begin{equation}\label{3.19}
\begin{aligned}
\frac{\|{\bf d}_k\|^2}{\left(g({\bf z}_k)^\top{\bf d}_k\right)^2} &= \frac{\left\|\frac{1}{N} \sum_{i=1}^N \beta_{k_i} {\bf d}_{k-i}\right\|^2}{\left(g({\bf z}_k)^\top{\bf d}_k\right)^2}-2\gamma_k \frac{g({\bf z}_k)^\top{\bf d}_k}{\left(g({\bf z}_k)^\top{\bf d}_k\right)^2}-\gamma_k^2 \frac{\|g({\bf z}_k)\|^2}{(g({\bf z}_k)^\top{\bf d}_k)^2} \\
&= \frac{\left\|\frac{1}{N} \sum_{i=1}^N \beta_{k_i} {\bf d}_{k-i}\right\|^2}{(g({\bf z}_k)^\top{\bf d}_k)^2}-\frac{2\gamma_k}{g({\bf z}_k)^\top{\bf d}_k}-\gamma_k^2 \frac{\|g({\bf z}_k)\|^2}{(g({\bf z}_k)^\top{\bf d}_k)^2} \\
&= \frac{\left\|\frac{1}{N} \sum_{i=1}^N \beta_{k_i} {\bf d}_{k-i}\right\|^2}{(g({\bf z}_k)^\top{\bf d}_k)^2}-\left( \frac{1}{\|g({\bf z}_k)\|} + \gamma_k \frac{\|g({\bf z}_k)\|}{g({\bf z}_k)^\top{\bf d}_k} \right)^2 + \frac{1}{\|g({\bf z}_k)\|^2}\\
&\leq \frac{\left\|\frac{1}{N} \sum_{i=1}^N \beta_{k_i} {\bf d}_{k-i}\right\|^2}{(g({\bf z}_k)^\top{\bf d}_k)^2} + \frac{1}{\|g({\bf z}_k)\|^2} \\
%&= \left( \frac{\frac{1}{N} \sum_{i=1}^N \beta_{k_i} \|{\bf d}_{k-i}\|}{|g({\bf z}_k)^\top{\bf d}_k|} \right)^2 + \frac{1}{\|g({\bf z}_k)\|^2} \\
&= \left( \frac{\frac{1}{N} \sum_{i=1}^N \beta_{k_i} \|{\bf d}_{k-i}\|}{\frac{1}{N} \sum_{i=1}^N \beta_{k_i} (\gamma_k\varphi_{k_i}-g({\bf z}_k)^\top{\bf d}_{k-i})} \right)^2 + \frac{1}{\|g({\bf z}_k)\|^2}.
\end{aligned}
\end{equation}
Noting that $\gamma_k \varphi_{k_i} \geq g({\bf z}_k)^\top{\bf d}_{k-i}+ \|g({\bf z}_k)\| \|{\bf d}_{k-i}\|$ and multiplying this by $\beta_{k_i} \geq 0$, we have
$$
\beta_{k_i} \left(\gamma_k \varphi_{k_i}-g({\bf z}_k)^\top{\bf d}_{k-i}\right) \geq \beta_{k_i} \|g({\bf z}_k)\| \|{\bf d}_{k-i}\|.
$$
Summing the above inequality over $i$, we obtain
$$
\sum_{i=1}^N \beta_{k_i} \left(\gamma_k\varphi_{k_i}-g({\bf z}_k)^\top{\bf d}_{k-i}\right) \geq \|g({\bf z}_k)\| \sum_{i=1}^N \beta_{k_i} \|{\bf d}_{k-i}\|.
$$
It follows from \eqref{3.18} that \(\sum_{i=1}^N \beta_{k_i} (\gamma_k\varphi_{k_i} -g({\bf z}_k)^\top{\bf d}_{k-i})> 0\), and we obtain
\begin{equation}\label{3.20}
\frac{\sum_{i=1}^N \beta_{k_i} \|{\bf d}_{k-i}\|}{\sum_{i=1}^N \beta_{k_i} (\gamma_k \varphi_{k_i} -g({\bf z}_k)^\top{\bf d}_{k-i})} \leq \frac{1}{\|g({\bf z}_k)\|}.
\end{equation}
By \eqref{3.19} and \eqref{3.20}, we have
\begin{equation}\label{3.28}
\frac{\|{\bf d}_k\|^2}{\left(g_1(\x_k,\y_k)^\top{\bf d}_k\right)^2} \leq \frac{1}{\|g({\bf z}_k)\|^2} + \frac{1}{\|g({\bf z}_k)\|^2} = \frac{2}{\|g({\bf z}_k)\|^2} \quad \text{for all } k.
\end{equation}
If \eqref{3.17} is not true, there exists a constant $c_1 > 0$ such that
\begin{equation}\label{3.29}
\|g({\bf z}_k)\| \geq c_1 \quad \text{for all } k.
\end{equation}
Therefore, it follows from \eqref{3.28} and \eqref{3.29} that
$$
\frac{(g({\bf z}_k)^\top{\bf d}_k)^2}{\|{\bf d}_k\|^2} \geq \frac{c_1^2}{2}.
$$
Thus, we obtain
$$
\sum_{k=0}^\infty \frac{\left(g({\bf z}_k)^\top{\bf d}_k\right)^2}{\|{\bf d}_k\|^2} = \infty.
$$
This contradicts Lemma \ref{lema3.4}, confirming that (\ref{3.17}) holds.
\qed

This theorem implies that for any choices of $\gamma_k,\varphi_{k_i}$ ($i=1,\ldots, N$) satisfy condition (\ref{3.6}), global convergence of our method is achieved.

\section{Numerical results}

In this section, we conduct some computational experiments to evaluate the efficacy and stability of the proposed method for solving problem (\ref{2.1}) or (\ref{2.5}). We compare the MGM method with the SS-HOPM in \cite{WQZ09} and the SIPM in \cite{WCW23}, which have been reported to be efficient for unconstrained optimization. All numerical experiments were implemented in MATLAB 9.0 on a personal computer with AMD Ryzen 7 4800H CPU 2.90GHz and 16 GB random-access memory (RAM).

As established in Theorem \ref{them3.2}, the global convergence of the method is guaranteed for any parameters $\gamma_k,\varphi_{k_i}$ satisfying conditions (\ref{3.6}). In our experiments, we first selected $\gamma_k$,  and then determined corresponding values of $\varphi_{k_i}$ for $i = 1,\ldots, N$ such that (\ref{3.6}) holds. For all methods, we set $\bar{\gamma}=10^{-15}$. In particular, we adopted the following two choices for $\gamma_k$.
\begin{enumerate}
   \item[(i)] $\gamma_k=1$ for all $k$;
   \item[(ii)]$\gamma_0=1$, and for $k\geq1$
$$
\begin{aligned}
\gamma_k &=
\begin{cases}
1, & \text{if }~ \dfrac{\w_{k-1}^\top \s_{k-1}}{\w_{k-1}^\top \w_{k-1}} < \bar{\gamma}, \\
\dfrac{\w_{k-1}^\top \s_{k-1}}{\w_{k-1}^\top \w_{k-1}}, & \text{otherwise}.
\end{cases}
%\quad  % 增加水平间距
%\gamma_k' &=
%\begin{cases}
%1, & \text{if } ~\dfrac{\w_{k-1}'^\top \s_{k-1}'}{\w_{k-1}'^\top \w_{k-1}'} < \bar{\gamma}, \\
%\dfrac{\w_{k-1}'^\top \s_{k-1}'}{\w_{k-1}'^\top \w_{k-1}'}, & \text{otherwise}.
%\end{cases}
\end{aligned}
$$
\end{enumerate}
To clarify the adaptive rules, we define
$$\s_{k-1}={\bf z}_k-{\bf z}_{k-1}, {\bf t}_{k-1}=g({\bf z}_k)-g({\bf z}_{k-1}).$$
The modified vectors are defined as
$$\w_{k-1}={\bf t}_{k-1}+\frac{\theta_{k-1}}{\s_{k-1}^\top\boldsymbol{\xi}_{k-1}},$$
where $\boldsymbol{\xi}_{k-1}$ is any vectors satisfying $\s_{k-1}^\top\boldsymbol{\xi}_{k-1}\neq0$, ensuring  that the division is well-defined. The history-dependent correction coefficient is given by
$$\theta_{k-1}=6(f({\bf z}_k)-f({\bf z}_{k-1}))+
3(g({\bf z}_k)-g({\bf z}_{k-1}))^\top\s_{k-1}.$$

The choice of $\gamma_k$ in (ii) is motivated by the sizing technique of the modified secant condition introduced by Zhang et al. \cite{ZX01,ZD99}. For a given
$\gamma_k$, we define $\varphi_{k_i}$ for $i = 1,\ldots,N$ as follows
$$
\varphi_{k_i} = \dfrac{\|g({\bf z}_k)\|\|{\bf d}_{k-i}\|+g({\bf z}_k)^\top {\bf d}_{k-i}+m+ n}{\gamma_k}.
$$
This choice of $\varphi_{k_i}$ satisfies condition (\ref{3.6}).
We denote the MGM method with $\gamma_k$ from (i) and (ii) as MGM-1 and MGM-2, respectively.

In the implementation of the MGM algorithm, we set parameters $\epsilon=10^{-6},\rho=0.1,\sigma=0.5$. The maximum number of allowed outer iterations was set to 2000.

\subsection{Analysis of $\gamma_k$ and $N$}
To evaluate the impact of $\gamma_k$ on algorithmic performance, we consider the number of iterations and CPU time as key metrics. We employ the performance profile introduced by \cite{DM02} by Dolan and Mor$\acute{e}$ to analyze the performance of the MGM. Let $Y$ and $W$ be the sets of methods and test problems, $n_y, n_w$ be the number of methods and test problems, respectively. The performance profile $\phi:\mathbb{R}\to [0,1]$ is defined for each $y\in Y$ and $w\in W $ such that $b_{w,y}>0$ represents the number of iterations (or CPU time) required to solve problem $w$ by method $y$.

Furthermore, the performance profile is given by
$$
\phi_y(\tau) = \frac{1}{n_w} size \left\{ w \in W : r_{w,y} \leq \tau \right\},
$$
where $\tau > 0 $, $\text{size}\{\cdot\}$ denotes the number of elements in a set, and $r_{w,y}$ is the performance ratio defined as:
$$
r_{w,y} = \frac{b_{w,y}}{\min \left\{ b_{w,y} : y \in Y \right\}}.
$$
Therefore, the performance profile is visualized by plotting the cumulative distribution function  $\phi_y$. Notably, $\phi_y(1)$ indicates the probability that the solver outperforms all other solvers. The right side of the image for $\phi_y$ shows the robustness associated with a solver.

To discuss the impact of the choice of $\gamma_k$ on algorithm performance, we compare MGM-1 and MGM-2 based on the number of iterations and CPU time for the same value of $N$. We test them with values $N=1,3,5,7,9$.
\begin{figure}[htbp]
    \centering
    % 第一行子图（Iterations）
    \begin{subfigure}[b]{0.19\textwidth}
        \centering
        \includegraphics[width=\textwidth]{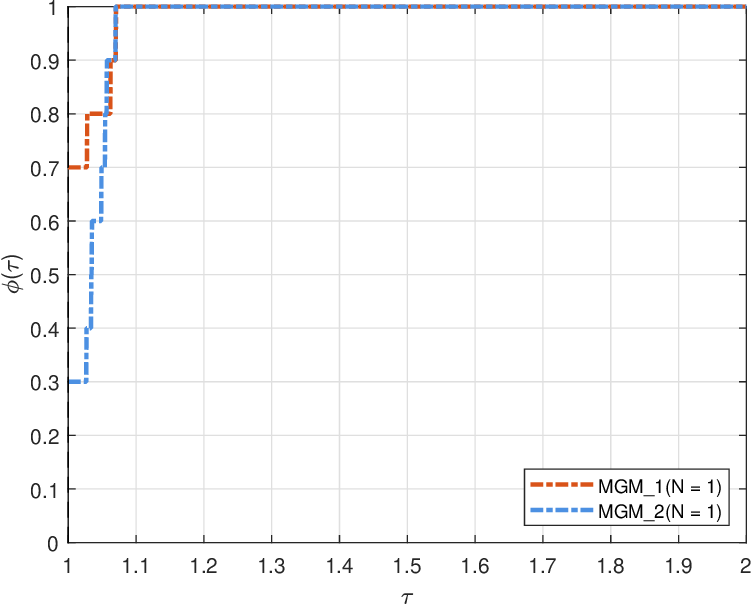} % 假设对应 N=1 的子图，按需改文件名
       % \caption*{MMG-I1 ($N=1$) vs MMG-I2 ($N=1$)}
    \end{subfigure}
    \begin{subfigure}[b]{0.19\textwidth}
        \centering
        \includegraphics[width=\textwidth]{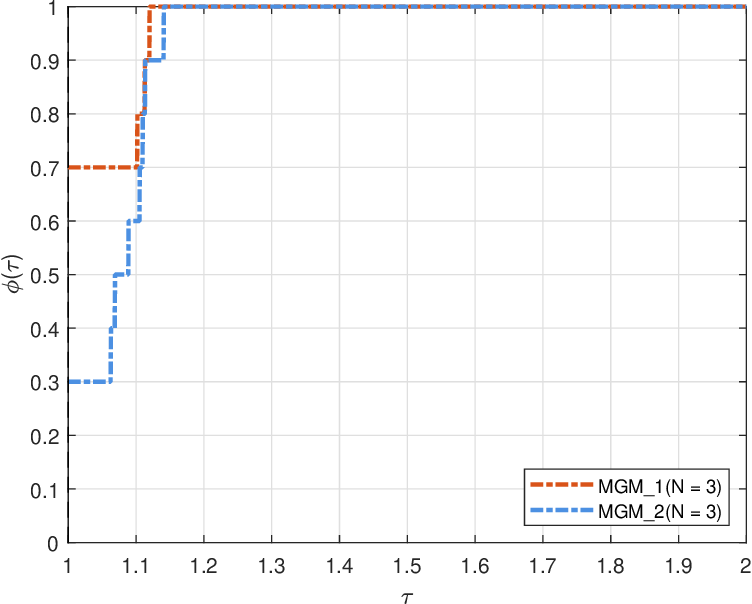} % 对应 N=3 的子图
        %\caption*{MMG-I1 ($N=3$) vs MMG-I2 ($N=3$)}
    \end{subfigure}
    \begin{subfigure}[b]{0.19\textwidth}
        \centering
        \includegraphics[width=\textwidth]{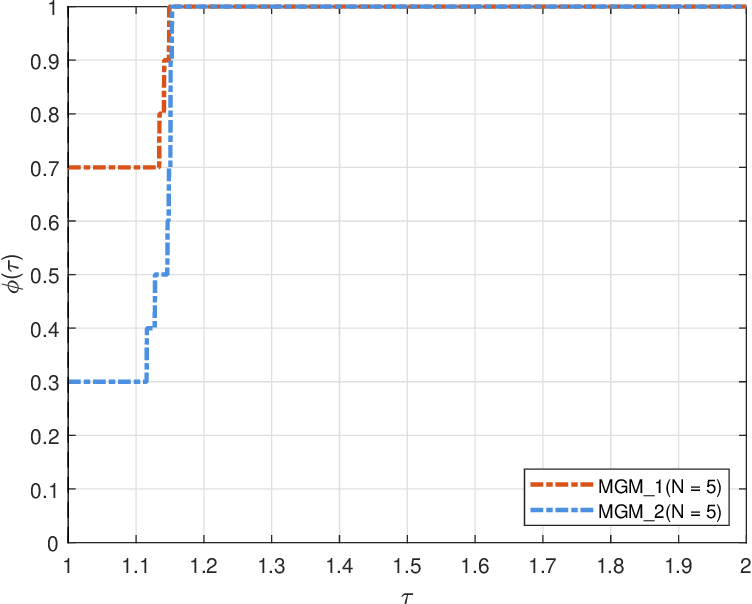} % 对应 N=5 的子图
        %\caption*{MMG-I1 ($N=5$) vs MMG-I2 ($N=5$)}
    \end{subfigure}
    \begin{subfigure}[b]{0.19\textwidth}
        \centering
        \includegraphics[width=\textwidth]{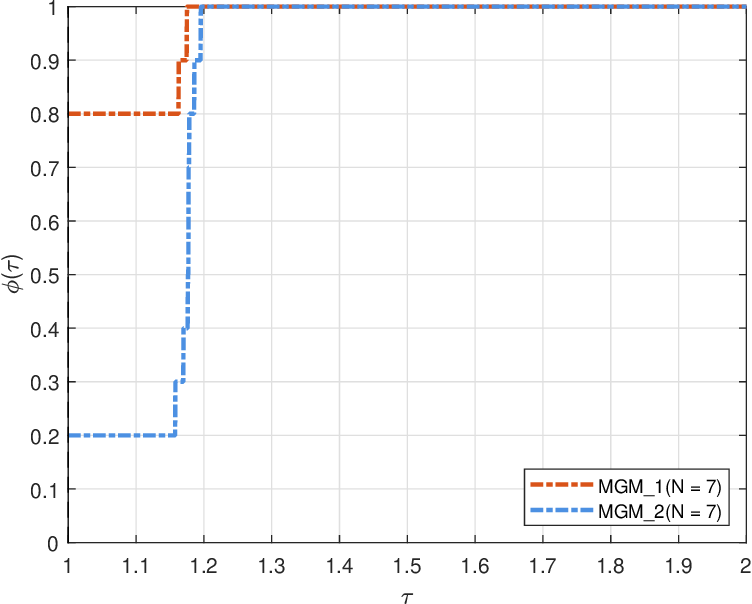} % 对应 N=7 的子图
        %\caption*{MMG-I1 ($N=7$) vs MMG-I2 ($N=7$)}
    \end{subfigure}
    \begin{subfigure}[b]{0.19\textwidth}
        \centering
        \includegraphics[width=\textwidth]{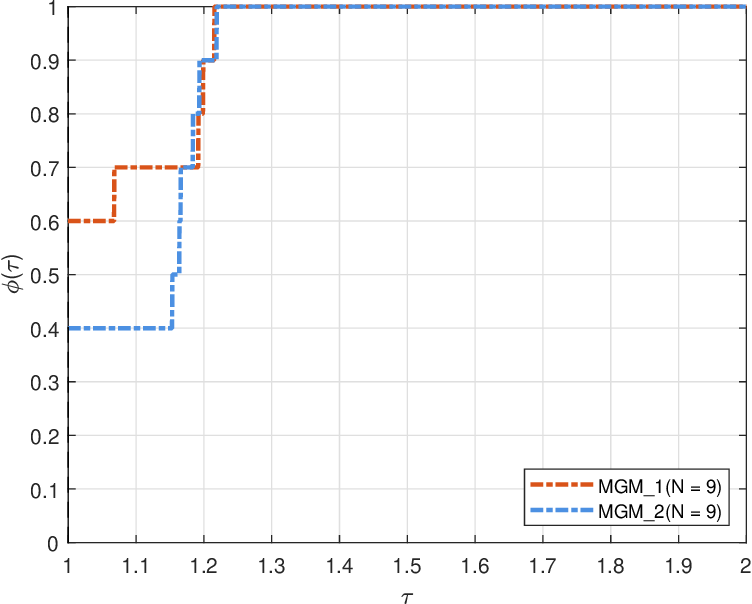} % 对应 N=9 的子图
        %\caption*{MMG-I1 ($N=9$) vs MMG-I2 ($N=9$)}
    \end{subfigure}
    \caption*{(a)~Iterations}

    % 第二行子图（Function evaluations）
    \vspace{0.5cm} % 增加垂直间距
    \begin{subfigure}[b]{0.19\textwidth}
        \centering
        \includegraphics[width=\textwidth]{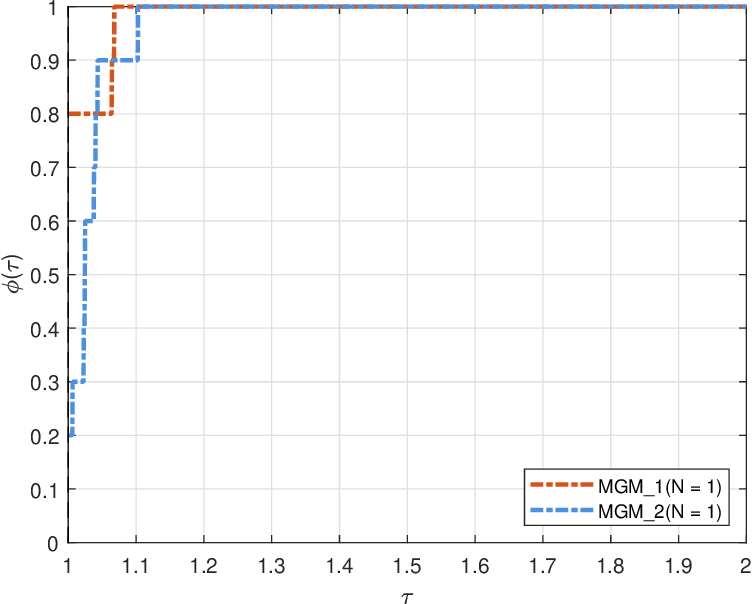} % 假设对应 N=1 函数评估的子图
        %\caption*{MMG-I1 ($N=1$) vs MMG-I2 ($N=1$)}
    \end{subfigure}
    \begin{subfigure}[b]{0.19\textwidth}
        \centering
        \includegraphics[width=\textwidth]{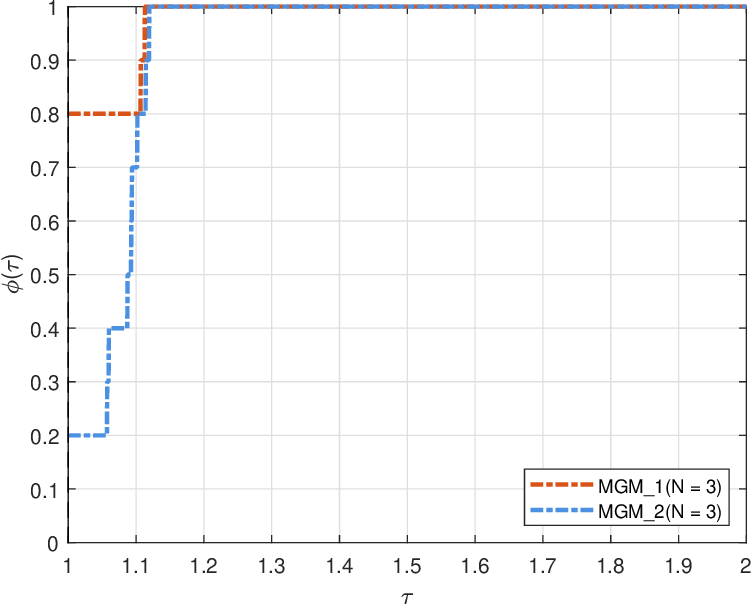} % 对应 N=3 函数评估的子图
        %\caption*{MMG-I1 ($N=3$) vs MMG-I2 ($N=3$)}
    \end{subfigure}
    \begin{subfigure}[b]{0.19\textwidth}
        \centering
        \includegraphics[width=\textwidth]{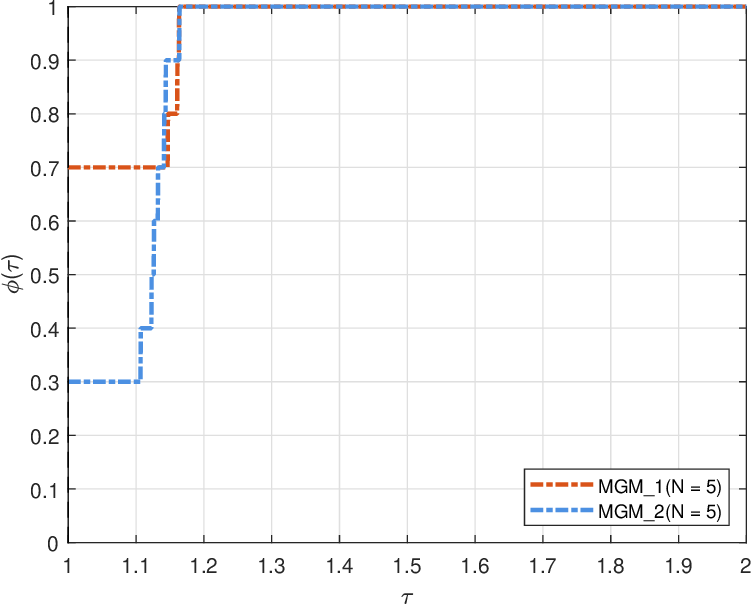} % 对应 N=5 函数评估的子图
        %\caption*{MMG-I1 ($N=5$) vs MMG-I2 ($N=5$)}
    \end{subfigure}
    \begin{subfigure}[b]{0.19\textwidth}
        \centering
        \includegraphics[width=\textwidth]{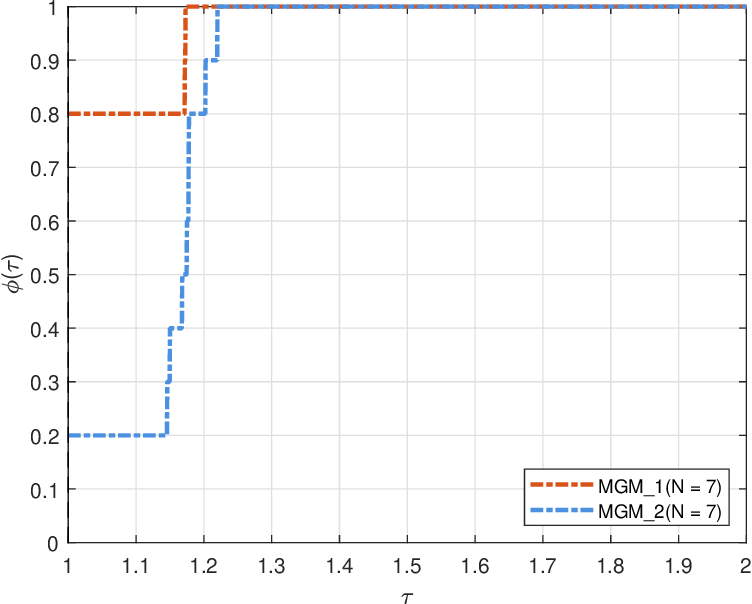} % 对应 N=7 函数评估的子图
        %\caption*{MMG-I1 ($N=7$) vs MMG-I2 ($N=7$)}
    \end{subfigure}
    \begin{subfigure}[b]{0.19\textwidth}
        \centering
        \includegraphics[width=\textwidth]{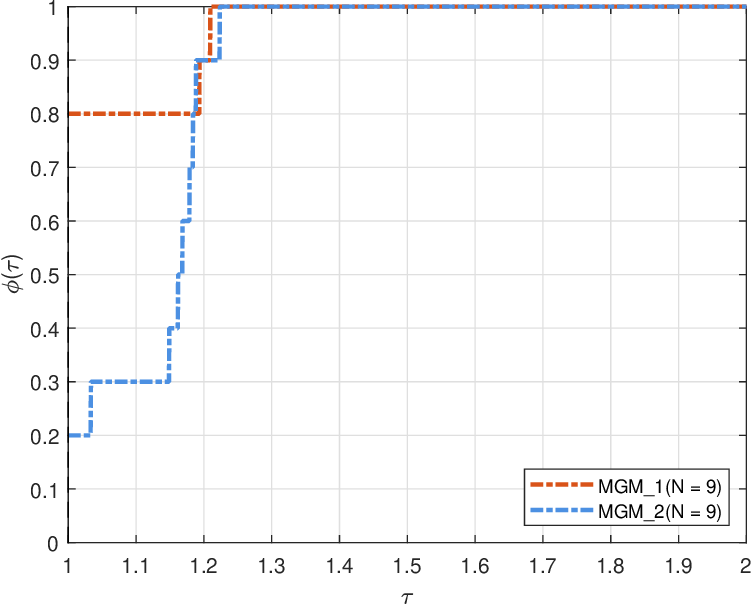} % 对应 N=9 函数评估的子图
        %\caption*{MMG-I1 ($N=9$) vs MMG-I2 ($N=9$)}
    \end{subfigure}
    \caption*{(b) CPU time}

    \caption{Performance profiles for MGM-1 and MGM-2 under the same $N$ value.}
    \label{fig1}
\end{figure}
Figure \ref{fig1} presents the performance profiles based on the number of iterations and CPU time. The results indicate that the performance of our methods is influenced by the parameters  $\gamma_k$ for the same value of $N$.

To study how a selection of $N$ value in MGM-1 and MGM-2 affect numerical performance, we conducted tests with $N=1,3,5,7,9$.
\begin{figure}[!ht]
    \centering
    % 子图 (a) Iterations
    \begin{subfigure}[b]{0.35\textwidth}
        \centering
        \includegraphics[width=\textwidth]{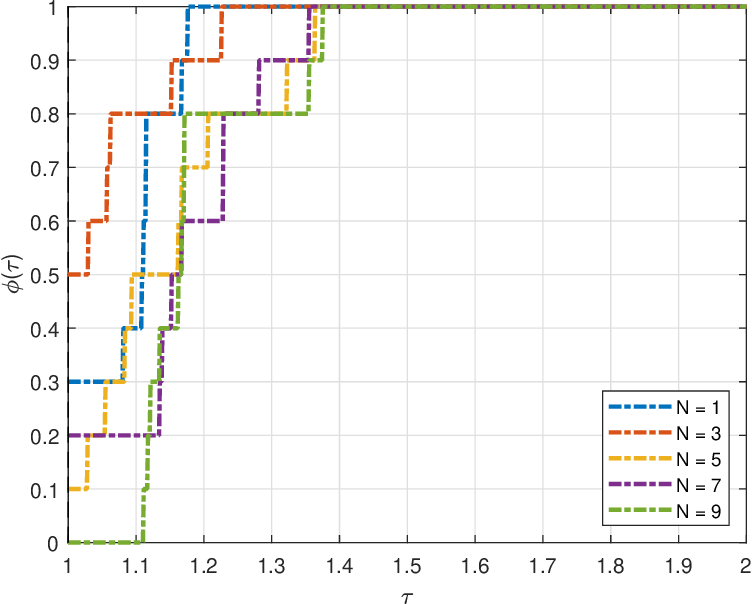}
        \caption{Iterations}
        \label{subfig:iterations}
    \end{subfigure}
    \hspace{15pt}
    % 子图 (b) CPU time
    \begin{subfigure}[b]{0.35\textwidth}
        \centering
        \includegraphics[width=\textwidth]{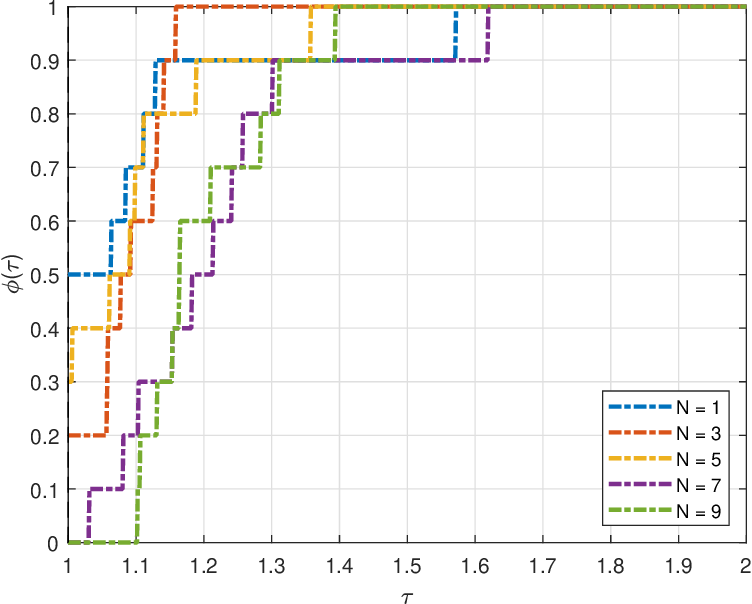}
        \caption{CPU time}
        \label{subfig:CPU time}
    \end{subfigure}
    \caption{Performance profiles for MGM-1 with different $N$ values.}
    \label{fig2}
    \vspace{1cm}
%\end{figure}
%\begin{figure}[htbp]
    \centering
    % 子图 (a) Iterations
    \begin{subfigure}[b]{0.35\textwidth}
        \centering
        \includegraphics[width=\textwidth]{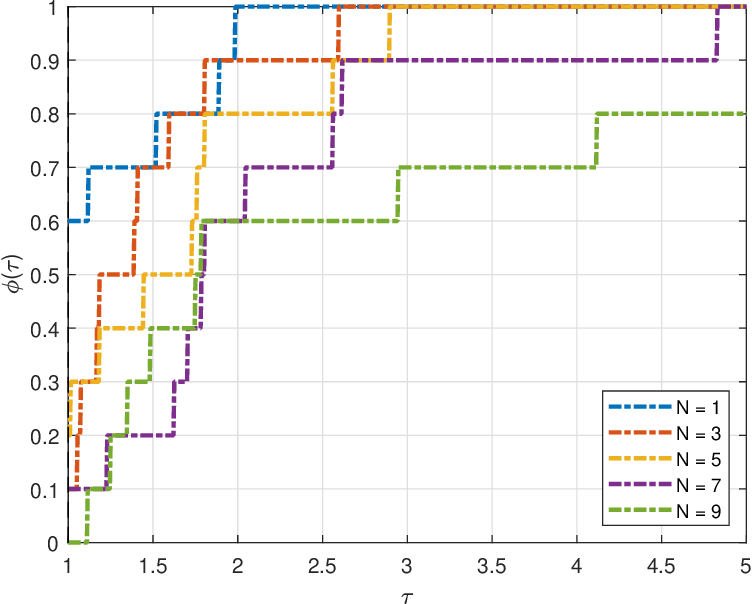}
        \caption{Iterations}
    \end{subfigure}
    \hspace{15pt}
    % 子图 (b) CPU time
    \begin{subfigure}[b]{0.35\textwidth}
        \centering
        \includegraphics[width=\textwidth]{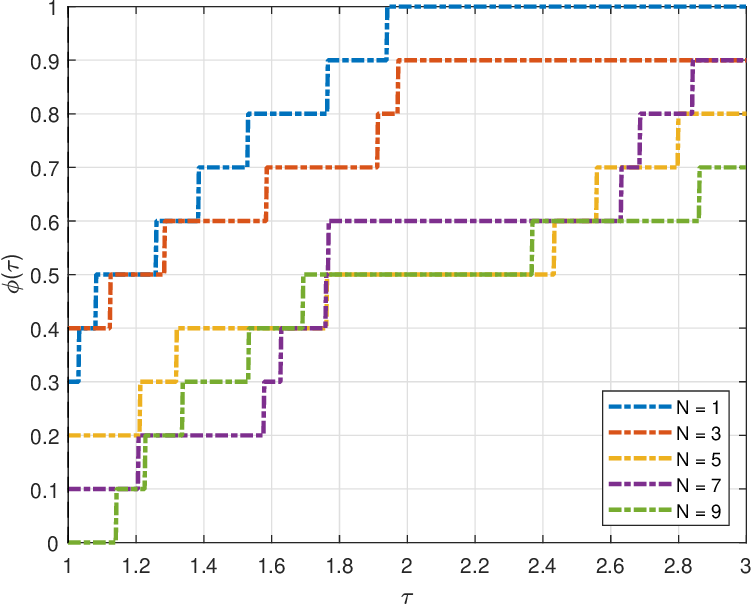}
        \caption{CPU time}
    \end{subfigure}
    \caption{Performance profiles for MGM-2 with different $N$ values.}
    \label{fig3}
\end{figure}
Figures 2 and 3 illustrate the performance profiles of MGM-1 and MGM-2 in terms of the number of iterations and CPU time, respectively. The results indicate that both MGM-1 and MGM-2 exhibit varying performance across different ranges depending on the chosen $N$ values. Therefore, the performance of MGM-1 and MGM-2 relies on the selection of parameter $N$. While it is challenging to determine the optimal choice theoretically, our experiments suggest that  $N=3$ for MGM-1 and $N=1$ for MGM-2 provide relatively robust performance.

\subsection{Analysis of all compared algorithms}
In this subsection, we present the results for SIPM, SS-HOPM, MGM-1 with $N=3$ and MGM-2 with $N=1$  to verify the performance of our method.

In Table 1, ``iter'' denotes the number of iterations, ``time'' represents the average CPU time by algorithms in seconds, ``$\lambda^\ast$''indicates the largest M-eigenvalue outputted by algorithms.
\begin{example}\label{ex1}
Let $\mathcal{A}= (a_{ijkl}) \in \mathbb{H}^{2\times2\times2\times 2}$ be a hierarchically symmetric tensor, whose entries are
\[
\begin{aligned}
a_{1111} &= 2, \quad a_{1211} = 3, \quad a_{2111} = 6, \quad a_{1121} = 6, \quad a_{1112} = 3, \quad a_{1212} = 2, \\
a_{2212} &= 10, \quad a_{1222} = 10, \quad a_{2222} = 5,
\end{aligned}
\]
and other $a_{ijkl} = 0$.
\end{example}

\begin{example}\label{ex2}
Consider the hierarchically symmetric tensor $\mathcal{A}=(a_{ijk\ell}) \in \mathbb{H}^{3\times3 \times3\times 3}$ with
\begin{align*}
a_{2222} = a_{1111} &= 196.6, \quad a_{3311} = a_{2233} = 83.2, \quad a_{2323} = a_{2322} = a_{3131} = a_{1331} = 54.7, \\
a_{2233} = a_{2323} &= -a_{1213} = -a_{2131} = -31.7, \quad a_{3333} = 110, \quad a_{2121} = a_{2121} = 64.4, \\
a_{2321} = a_{1232} &= -a_{1331} = -a_{1331} = -25.3, \quad a_{3112} = a_{1321} = 44.8, \\
a_{2132} = a_{1223} &= -35.84, \quad a_{1122} = 132.2,
\end{align*}
and other $a_{ijk\ell} = 0$.
\end{example}

\begin{example}\label{ex3}
Consider the tensor whose entries are uniformly generated in $(-1,1)$:
$$
\A(:,:,1,1) = \begin{pmatrix}
-0.9727 & 0.3169 & -0.3437 \\
-0.6332 & -0.7866 & 0.4257 \\
-0.3350 & -0.9896 & -0.4323
\end{pmatrix},
$$

$$
\A(:,:,2,1) = \begin{pmatrix}
-0.6332 & -0.7866 & 0.4257 \\
0.7387 & 0.6873 & -0.3248 \\
-0.7986 & -0.5988 & -0.9485
\end{pmatrix},
$$

$$
\A(:,:,3,1) = \begin{pmatrix}
-0.3350 & -0.9896 & -0.4323 \\
-0.7986 & -0.5988 & -0.9485 \\
0.5853 & 0.5921 & 0.6301
\end{pmatrix},
$$

$$
\A(:,:,1,2) = \begin{pmatrix}
0.3169 & 0.6158 & -0.0184 \\
-0.7866 & 0.0160 & 0.0085 \\
-0.9896 & -0.6663 & 0.2559
\end{pmatrix},
$$

$$
\A(:,:,2,2) = \begin{pmatrix}
-0.7866 & 0.0160 & 0.0085 \\
0.6873 & 0.5160 & -0.0216 \\
-0.5988 & 0.0411 & 0.9857
\end{pmatrix},
$$

$$
\A(:,:,3,2) = \begin{pmatrix}
-0.9896 & -0.6663 & 0.2559 \\
-0.5988 & 0.0411 & 0.9857 \\
0.5921 & -0.2907 & -0.3881
\end{pmatrix},
$$

$$
\A(:,:,1,3) = \begin{pmatrix}
-0.3437 & -0.0184 & 0.5649 \\
0.4257 & 0.0085 & -0.1439 \\
-0.4323 & 0.2559 & 0.6162
\end{pmatrix},
$$
$$
\A(:,:,2,3) = \begin{pmatrix}
0.4257 & 0.0085 & -0.1439 \\
-0.3248 & -0.0216 & -0.0037 \\
-0.9485 & 0.9857 & -0.7734
\end{pmatrix},
$$

$$
\A(:,:,3,3) = \begin{pmatrix}
-0.4323 & 0.2559 & 0.6162 \\
-0.9485 & 0.9857 & -0.7734 \\
0.6301 & -0.3881 & -0.8526
\end{pmatrix}.
$$
\end{example}

\begin{table}[htbp]
    \centering
    \caption{The numerical results of Examples \ref{ex1}-\ref{ex3}.}
    \label{tab1}
    \tabcolsep=0.5cm
  \renewcommand\arraystretch{1.0}
    \begin{tabular}{@{}l l l l l l @{}} % @{} 去除表格左右默认空白
      \hline
        $$&          $$&         MGM-1&      MGM-2&     SIPM&     SS-HOPM\\
        Example & $\lambda^\ast$ & iter/time     &iter/time      &iter/time  &iter/time \\
        \hline% booktabs 宏包的中间线条
        4.1 & 13.8616 & 5/0.0312 & 10/0.0899 & 11/0.0595& 15/0.1286  \\
        4.2 & 318 & 11/0.0236 & 25/0.0521 & 34/0.0677 & 40/0.1459 \\
        4.3 & 2.3227 & 10/0.1233 & 29/0.2221 & 30/0.2767 & 35/0.3292 \\
        \hline % booktabs 宏包的底部线条
    \end{tabular}
\end{table}
Table \ref{tab1} demonstrates the superior performance of MGM-1 for computing M-eigenvalues for Examples \ref{ex1}-\ref{ex3}. MGM-1 demonstrates significantly faster convergence compared to other methods, achieving fewer iterations and lower computational costs.

%All algorithms share the same start points and stopping criteria.

\begin{example}\label{ex4}
Suppose that $\mathcal{A}\in\mathbb{R}^{n\times n\times n\times n}$ is a hierarchically symmetric tensor with entries uniformly generated in (0,1), i.e., $\mathcal{A}\sim U(0,1)$.
\end{example}

Since the elements are generated in a random way, we investigate the averaged performance of the four methods by generating 10 groups of data for each case. Specifically, we investigate six cases on the problems' dimension, i.e., $n=\{5,10,15,20,25,30\}$. All methods start with the same random vectors, which are drawn from a standard normal distribution.

\begin{figure}[!ht]
    \centering
    % 子图 (a) Iterations
    \begin{subfigure}[b]{0.44\textwidth}
        \centering
        \includegraphics[width=\textwidth]{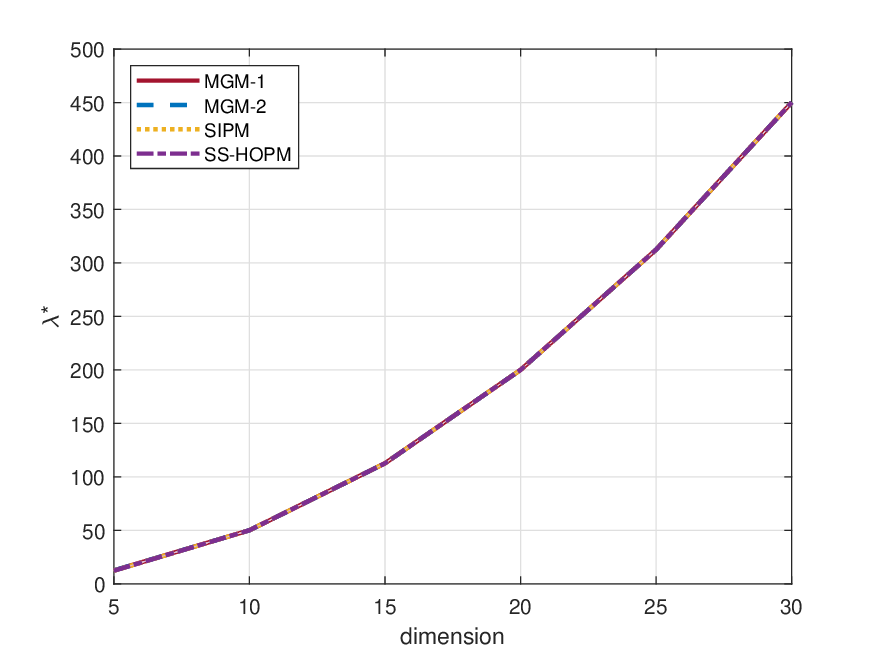}
        \caption{M-eigenvalues}
    \end{subfigure}
    \hspace{15pt}
    % 子图 (b) CPU time
    \begin{subfigure}[b]{0.44\textwidth}
        \centering
        \includegraphics[width=\textwidth]{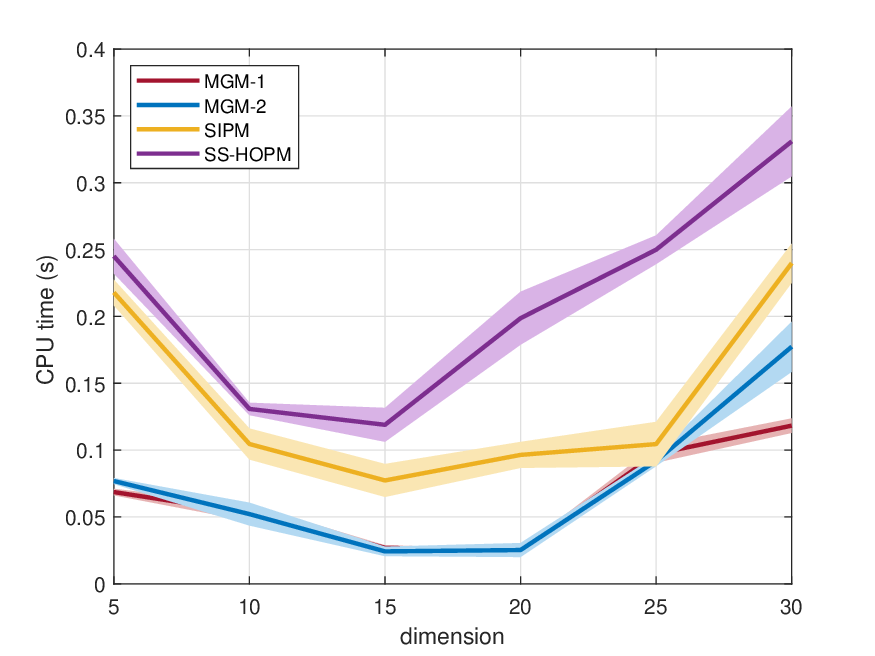}
        \caption{CPU time}
    \end{subfigure}
    \caption{Full performance of both MGM-1, MGM-2, SIPM and SS-HOPM in terms of the averaged values and standard derivation of 10 random tests.}
    \label{fig4}
\end{figure}

Figure \ref{fig4} shows the averaged M-eigenvalues and the CPU time, represented by solid lines. The shaded regions, colored accordingly, represent the distribution of values over 10 random experiments, with the width of the shaded regions indicating the standard deviation. A wider shaded region suggests a more unstable performance of the method. The two plots in Figure \ref{fig4} sufficiently support that the MGM-1 performs quite stable.

\begin{example}\label{ex5}
Let $\mathcal{A}\in\mathbb{R}^{m\times n\times m\times n}$ be a hierarchically symmetric tensor generated randomly with elements in (-5,5). In this experiment, we set $N=3$ for MGM-1 and $N=1$ for MGM-2.
\end{example}

The results are presented in Table \ref{tab2}, which show that MGM-1 outperforms its competitors in terms of efficiency and computational cost across all test cases.

\medskip
\begin{table}[ht]
  \caption{The numerical results of Example \ref{ex5}}
  \label{tab2}
  \begin{center}
  \tabcolsep=0.5cm
  \renewcommand\arraystretch{1.0}
    {\begin{tabular}[c]
      {  c  c  c  c c  c}
      \hline
        $$  &       $ $ &      MGM-1&    MGM-2&     SIPM &        SS-HOPM\\
         $(m,n)$  & $\lambda^\ast$ &  iter/time& iter/time&  iter/time &  iter/time\\
        \hline
         (12,18) &  14.9986 &     13/0.2352 &    15/0.3576&    21/0.5373  &     42/0.8522   \\
         (30,18) &  18.9550 &     24/1.5623 &    33/1.2134&    53/1.9543  &     84/2.3742    \\
         (30,30) &  23.2705 &     43/1.7412 &    40/1.5431&    72/2.5241  &     93/3.6652   \\
         (50,60) &  30.0465 &     51/3.2582 &    58/3.4671&    68/4.2521  &     153/6.3823   \\
         \hline
    \end{tabular}}
  \end{center}
\end{table}
\medskip

\section{Conclusions}

In this paper, we focused on computing extremal M-eigenvalue of fourth order hierarchically symmetric tensors by transforming the M-eigenvalue problem into an unconstrained optimization framework.
Based on a new reformulaton, we proposed a memory gradient method to find an approximate solution of the problem. Furthermore, global convergence of the proposed method is established and several numerical examples verify the performance of the algorithm.
\medskip

%\noindent{\bf Acknowledgment} We are very grateful to the editor and the anonymous reviewers for their constructive comments and meaningful
%suggestions on our manuscript.


\begin{thebibliography}{99}

\bibitem{AG91} Auchmuty G, {\it Globally and rapidly convergent algorithms for symmetric eigenproblems}, SIAM Journal on Matrix Analysis and Applications, (1991), 12(4): 690-706.

\bibitem{AN07} Andrei N, {\it A scaled BFGS preconditioned conjugate gradient algorithm for unconstrained optimization}, Applied Mathematics Letters, (2007), 20(6): 645-650.

\bibitem{AN207} Andrei N, {\it Scaled memoryless BFGS preconditioned conjugate gradient algorithm for unconstrained optimization}, Optimization Methods and Software, (2007), 22(4): 561-571.

\bibitem{B1970} Backus G, {\it A geometrical picture of anisotropic elastic tensors}, Reviews of geophysics, (1970), 8: 633-671.


\bibitem{BV08} Bloy L, Verma R, {\it On Computing the Underlying Fiber Directions From the Diffusion Orientation Distribution Function in Medical Image Computing and Computer-Assisted Intervention}, Springer: Berlin/Heidelberg, Germany, (2008), pp. 1-8.

\bibitem{CQ10} Chang K, Qi L, Zhou G, {\it Singular values of a real rectangular tensor}, Journal of Mathematical Analysis and Applications, (2010), 370(1): 284-294.

\bibitem{CCW20} Che H, Chen H, Wang Y, {\it On the M-eigenvalue estimation of fourth order partially symmetric tensors}, Journal of Industrial and Management Optimization, (2020), 16(1): 309-324.

%\bibitem{CQY13} Chen Z, Qi L, Yang Q, {\it The solution methods for the largest eigenvalue (singular value) of nonnegative tensors and convergence analysis}, Linear Algebra and its Applications, (2013), 439(12): 3713-3733.
%
\bibitem{CDC07} Chirita S, Danescu A, Ciarletta, {\it On the strong ellipticity of the anisotropic linearly elastic materials}, Journal of Elasticity, (2007), 87: 1-27.

\bibitem{CHW22} Chen H, He H, Wang Y, et al, {\it An efficient alternating minimization method for fourth degree polynomial optimization}, Journal of Global Optimization, (2022), 82: 83-103.

\bibitem{DLM07} Dahl G, Leinaas J, Myrheim J, {\it A tensor product matrix approximation problem in quantum physics}, Linear Algebra and its Applications, (2007), 420(2-3): 711-725.

\bibitem{DDV00} De L, De M, Vandewalle J, {\it On the best rank-1 and rank-$(R_1, R_2, \ldots, R_N)$ approximation of higher-order tensors}, SIAM journal on Matrix Analysis and Applications, (2000), 21(4): 1324-1342.

\bibitem{DQW13} Ding W, Qi L, Wei Y, {\it  M-tensors and nonsingular M-tensors}, Linear Algebra and Its Applications, (2013), 439(10): 3264-3278.

\bibitem{DM02} Dolan E, Mor$\acute{e}$ J, {\it Benchmarking optimization software with performance profiles}, Mathematical programming, (2002), 91(2): 201-213.

\bibitem{EPR35} Einstein A, Podolsky B, Rosen N, {\it Can quantum-mechanical description of physical reality be considered complete} ? Physical Review, (1935), 47(10): 777-780.

\bibitem{FHR1996} Fj\ae r  E,  Holt R M,  Rathore J  S, {\it Seismic Anisotropy, Chapter 3: Representation and Approximation of Elastic Tensors}, Society of Exploration Geophysicists (1996)


\bibitem{GC19}  Gao L, Cao Z, Wang G, {\it Almost sure stability of discrete-time nonlinear Markovian jump delayed systems with impulsive signals}, Nonlinear Analysis: Hybrid Systems, (2019), 34: 248-263.

\bibitem{GL19} Gao L, Luo F, Yan Z, {\it Finite-time annular domain stability of impulsive switched systems: mode-dependent parameter approach}, International Journal of Control, (2019), 92(6): 1381-1392.

\bibitem{GZ19} Gao L, Zhang M, Yao X, {\it Stochastic input-to-state stability for impulsive switched stochastic nonlinear systems with multiple jumps}, International Journal of Systems Science, (2019), 50(9): 1860-1871.

\bibitem{GM72} Gurtin M, {\it The Linear Theory of Elasticity in: Handbuch der Physik}, Springer, Berlin, (1972).

\bibitem{HDH09} Han D, Dai H, Qi L, {\it Conditions for strong ellipticity of anisotropic elastic materials}, Journal of Elasticity, (2009), 97: 1-13.

\bibitem{HL13} Han L, {\it An unconstrained optimization approach for finding real eigenvalues of even order symmetric tensors}, Numerical Algebra Control and Optimization, 3 (2013), 583-599.

\bibitem{HC15} Hao C, Cui C, Dai Y, {\it A sequential subspace projection method for extreme Z-eigenvalues of supersymmetric tensors}, Numerical Linear Algebra with Applications, (2015), 22(2): 283-298.

\bibitem{HCD15} Hao C, Cui C, Dai Y, {\it A feasible trust-region method for calculating extreme Z-eigenvalues of symmetric tensors}, (2015).
\bibitem{HLX21} He J, Liu Y, Xu G, {\it New M-eigenvalue inclusion sets for fourth order partially symmetric tensors with applications}, Bulletin of the Malaysian Mathematical Sciences Society, (2021), 44(6): 3929-3947.
\bibitem{HLW20} He J, Li C, Wei Y, {\it M-eigenvalue intervals and checkable sufficient conditions for the strong ellipticity}, Applied Mathematics Letters, (2020), 102: 106137.
\bibitem{H1994} Helbig K, {\it  Representation and Approximation of Elastic Tensors},  as poster PI-1 at the 6th International Workshop on Seismic Anisotropy, Trondheim, (1994), July: 3-8.


\bibitem{HQ18} Huang Z, Qi L, {\it Positive definiteness of paired symmetric tensors and elasticity tensors}, Journal of Computational and Applied Mathematics, (2018), 338: 22-43.

\bibitem{HLQ13} Hu S, Li G, Qi L, et al, {\it Finding the maximum eigenvalue of essentially nonnegative symmetric tensors via sum of squares programming}, Journal of Optimization Theory and Applications, (2013), 158(3): 717-738.

\bibitem{HXL17} He J, Xu G, Liu Y, {\it Some inequalities for the minimum M-eigenvalue of elasticity M-tensors}, Journal of Industrial and Management Optimization, (2017), 13(5): 1-11.

\bibitem{KR02} Kofidis E, Regalia P, {\it On the best rank-1 approximation of higher-order supersymmetric tensors}, SIAM Journal on Matrix Analysis and Applications, (2002), 23(3): 863-884.

\bibitem{LL15} Li C, Li Y, {\it Double B-tensors and quasi-double B-tensors}, Linear algebra and its applications, (2015), 466: 343-356.

\bibitem{LLK14} Li C, Li Y, Kong X, {\it New eigenvalue inclusion sets for tensors}, Numerical Linear Algebra with Applications, (2014), 21(1): 39-50.

\bibitem{LWZ14} Li C, Wang F, Zhao J, {\it Criterions for the positive definiteness of real supersymmetric tensors}, Journal of Computational and Applied Mathematics, (2014), 255: 1-14.

\bibitem{LG13} Li G, Qi L, Yu G, {\it The Z-eigenvalues of a symmetric tensor and its application to spectral hypergraph theory},  Numerical Linear Algebra with Application, (2013), 20, 1001-1029.

\bibitem{LLL19} Li S, Li C, Li Y, {\it M-eigenvalue inclusion intervals for a fourth order partially symmetric tensor}, Journal of Computational and Applied Mathematics, (2019), 356: 391-401.

\bibitem{MC69} Miele A, Cantrell J, {\it Study on a memory gradient method for the minimization of functions}, Journal of Optimization Theory and Applications, (1969), 3(6): 459-470.

\bibitem{NY06} Narushima Y, Yabe H, {\it Global convergence of a memory gradient method for unconstrained optimization}, Computational Optimization and Applications, (2006), 35(3): 325-346.

\bibitem{NQ15} Ni Q, Qi L, {\it A quadratically convergent algorithm for finding the largest eigenvalue of a nonnegative homogeneous polynomial map}, Journal of Global Optimization, (2015), 61(4): 627-641.

\bibitem{NW99} Nocedal J, Wright S, {\it Numerical optimization}, New York, NY: Springer New York, (1999).

\bibitem{PC02} Padovani C, {\it Strong Ellipticity of Transversely Isotropic Elasticity Tensors}, Meccanica, (2002), 37(6):515-525.

\bibitem{PS88} Peressini A, Sullivan F E, Uhl J J, {\it The mathematics of nonlinear programming}, New York: Springer-Verlag, (1988).

\bibitem{QDH09} Qi L, Dai H, Han D, {\it Conditions for strong ellipticity and M-eigenvalues}, Frontiers of Mathematics in China, (2009), 4(2):349-364.

\bibitem{QL17}Qi L, Luo Z, {\it Tensor analysis: spectral theory and special tensors}, Society for Industrial and Applied Mathematics, (2017).

\bibitem{QY10} Qi L, Yu G, Wu E, {\it Higher order positive semidefinite diffusion tensor imaging}, SIAM Journal on Imaging Sciences, (2010), 3(3): 416-433.

\bibitem{S35} Schr$\ddot{o}$dinger E, {\it Die gegenw$\ddot{a}$rtige situation in der quantenmechanik}, Naturwissenschaften, (1935), 23(50): 844-849.

\bibitem{WCW23} Wang C, Chen H, Wang Y, {\it An alternating shifted inverse power method for the extremal eigenvalues of fourth order partially symmetric tensors}, Applied Mathematics Letters, (2023), 141: 108601.

\bibitem{WQZ09} Wang Y, Qi L, Zhang X, {\it A practical method for computing the largest M-eigenvalue of a fourth order partially symmetric tensor}, Numerical Linear Algebra with Applications, (2009), 16(7): 589-601.

\bibitem{WSL20} Wang G, Sun L, Liu L, {\it M-eigenvalues-based sufficient conditions for the positive definiteness of fourth order partially symmetric tensors}, Complexity 2020, (2020), 3:1-8.

\bibitem{WW03} Walton J, Wilber J, {\it Sufficient conditions for strong ellipticity for a class of anisotropic materials}, International Journal of Non-Linear Mechanics, (2003), 38(4):441-455.

\bibitem{YY11} Yang Q, Yang Y, {\it Further results for Perron-Frobenius theorem for nonnegative tensors II}, SIAM Journal on Matrix Analysis and Applications, (2011), 32(4): 1236-1250.

\bibitem{ZX01} Zhang J, Xu C, {\it Properties and numerical performance of quasi-Newton methods with modified quasi-Newton equations}. Journal of Computational and Applied Mathematics, (2001), 137(2): 269-278.

\bibitem{ZD99} Zhang J, Deng N, Chen L, {\it New quasi-Newton equation and related methods for unconstrained optimization}, Journal of Optimization Theory and Applications, (1999), 102(1): 147-167.
\bibitem{ZO70} Zoutendijk G, {\it Nonlinear programming, computational methods in Integer and nonlinear programming},  J. Abadie, (ed.), North-Holland, Amsterdam, (1970): 37-86.




%\bibitem{GW18} Gu Y, Wu W, {\it Partially symmetric nonnegative rectangular tensors and copositive rectangular tensors}, Journal of Industrial and Management Optimization, (2018), 15(2): 775-789.
%
%\bibitem{HDH09} Han D, Dai H, Qi L, {\it Conditions for strong ellipticity of anisotropic elastic materials}, Journal of Elasticity, (2009), 97: 1-13.
%
%\bibitem{HLW20} He J, Li C, Wei Y, {\it M-eigenvalue intervals and checkable sufficient conditions for the strong ellipticity}, Applied Mathematics Letters, (2020), 102: 106137.
%
%\bibitem{HLK16} He J, Liu Y, Ke H, {\it Bound for the largest singular value of nonnegative rectangular tensors}, Open Mathematics, (2016), 14(1): 761-766.
%
%\bibitem{HLX19} He J, Liu Y, Xu G, {\it V-singular values of rectangular tensors and their applications}, Journal of Inequalities and Applications, (2019): 1-15.
%
%
%\bibitem{KS75} Knowles J K, Sternberg E, {\it On the ellipticity of the equations of nonlinear elastostatics for a special material}, Journal of Elasticity, (1975), 5(3-4): 341-361.
%
%\bibitem{KS76} Knowles J K, Sternberg E, {\it On the failure of ellipticity of the equations for finite elastostatic plane strain}, Archive for Rational Mechanics and Analysis, (1976), 63(4): 321-336.
%
%\bibitem{KM11} Kolda T G, Mayo J R, {\it Shifted power method for computing tensor eigenpairs}, SIAM Journal on Matrix Analysis and Applications, (2011), 32(4): 1095-1124.
%
%\bibitem{KM14} Kolda T G, Mayo  J R, {\it An adaptive shifted power method for computing generalized tensor eigenpairs}, SIAM Journal on Matrix Analysis and Applications, (2014), 35: 1563-1581.
%
%\bibitem{LL16} Li C, Li Y, {\it An eigenvalue localization set for tensors with applications to determine the positive (semi-)definiteness
%of tensors}, Linear Multilinear Algebra, (2016), 64: 587-601.
%
%\bibitem{LQ13} Ling C, Qi L, {\it $l^{k,s}$-Singular values and spectral radius of rectangular tensors}, Frontiers of Mathematics in China, (2013), 8: 63-83.


\end{thebibliography}
\end{document}